 \documentclass[11pt,a4paper,reqno]{amsart}
\usepackage[utf8]{inputenc}
\usepackage{amsmath, amsthm}
\usepackage{amsfonts}
\usepackage{amssymb}
\usepackage{graphicx}
\usepackage[left=2.5cm,right=2.5cm,top=2cm,bottom=2cm]{geometry}
\usepackage{amscd}
\usepackage[all]{xy}
\usepackage{hyperref}
\usepackage{gensymb}
\usepackage{enumerate}
\usepackage{csquotes}
\usepackage{comment}
\usepackage{hhline}
\usepackage{float}
\usepackage{subfigure}
\usepackage{color}

\usepackage{caption} 
\theoremstyle{plain}
\newtheorem{theorem}{Theorem}[section]
\newtheorem{Proposition}[theorem]{Proposition}

\newtheorem{Corollary}[theorem]{Corollary}

\newtheorem{example}[theorem]{Example}

\begin{document}

\graphicspath{ {PicturesArxiv/} } 

\title{Geometric methods for efficient planar swimming of copepod nauplii}

\author{Jonas Balisacan}
\address{University of Hawai`i at M\={a}noa, 2565 McCarthy Mall,  Honolulu, HI 96822, USA}
\email{jrab@hawaii.edu}

\author{Monique Chyba}
\address{University of Hawai`i at M\={a}noa, 2565 McCarthy Mall,  Honolulu, HI 96822, USA}
\email{chyba@hawaii.edu}

\author{Corey Shanbrom}
\address{California State University, Sacramento, 6000 J St., Sacramento, CA 95819, USA}
\email{corey.shanbrom@csus.edu}

\author{George Wilkens}
\address{University of Hawai`i at M\={a}noa, 2565 McCarthy Mall,  Honolulu, HI 96822, USA}
\email{grw@hawaii.edu}

\thanks{M.C. is partially supported by the Simons Foundation grant number 359510.}

\date{\today}


\begin{abstract}
Copepod nauplii are larval crustaceans with important ecological functions.  Due to their small size, they experience an environment of low Reynolds number within their aquatic habitat.  Here we provide a mathematical model of a swimming copepod nauplius with two legs moving in a plane.  This model allows for both rotation and two-dimensional displacement by periodic deformation of the swimmer's body.  The system is studied from the framework of optimal control theory, with a simple cost function designed to approximate the mechanical energy expended by the copepod.  We find that this model is sufficiently realistic to recreate behavior similar to those of observed copepod nauplii, yet much of the mathematical analysis is tractable.  In particular, we show that the system is controllable, but there exist singular configurations where the degree of non-holonomy is non-generic.  We also partially characterize the abnormal extremals and provide explicit examples of families of abnormal curves. Finally, we numerically simulate normal extremals and observe some interesting and surprising phenomena.
\end{abstract}

\maketitle

\section{Introduction}

Microcrustaceans known as copepods are one of the most abundant animals on Earth. They are a type of zooplankton that serve as an important link in the marine food web. As they are prey for many larger aquatic creatures, they must adapt strategies that help them maximize their survivability. Observations have shown that different copepods have adapted different types of movement to efficiently forage for food and evade predators \cite{Bradley, Bruno,  Paffenhofer, copepod}.
Figure \ref{fig: copepod photo} shows a nauplius of the copepod species \textit{Bestiolina similis}.

In the world of microorganisms, water becomes a very viscous fluid in which movement produces negligible inertia. 
This is known as a low Reynolds number environment, as the Reynolds number $\mathcal{R}$ represents the ratio between the inertial force due to momentum and the viscous force experienced from the resistance of the liquid.
An object swimming in some fluid experiences the Reynolds number $\mathcal{R} = \frac{a v \rho}{\eta}$, where $a$ is the characteristic dimension of the object, $v$ is the velocity of the object, $\rho$ is the density of the fluid, and $\eta$ is the fluid viscosity.  In water, organisms as small as bacteria have a Reynolds number of approximately $10^{-6}$ to $10^{-4}$, whereas humans have a Reynolds number of approximately $10^{6}$.   %Note that water has a viscosity of about $\SI{0.001}{\pascal\second}$ at $\SI{20}{\degreeCelsius}$ and molasses has a viscosity of between $\SI{5}{\pascal\second}$ and $\SI{10}{\pascal\second}$.
Nauplii of the paracalanid copepod \textit{Bestiolina similis}, as shown in Figure \ref{fig: copepod photo}, have lengths 70–200 $\mu$m and swim at Reynolds numbers of $10^{-1}-10^1$ (see \cite{Lenz} and the references therein).
To put swimming at low Reynolds number in perspective, humans have a Reynolds number around $10^{2}$ when swimming in molasses.

\begin{figure}
    \centering
\includegraphics[width=8.5 cm]{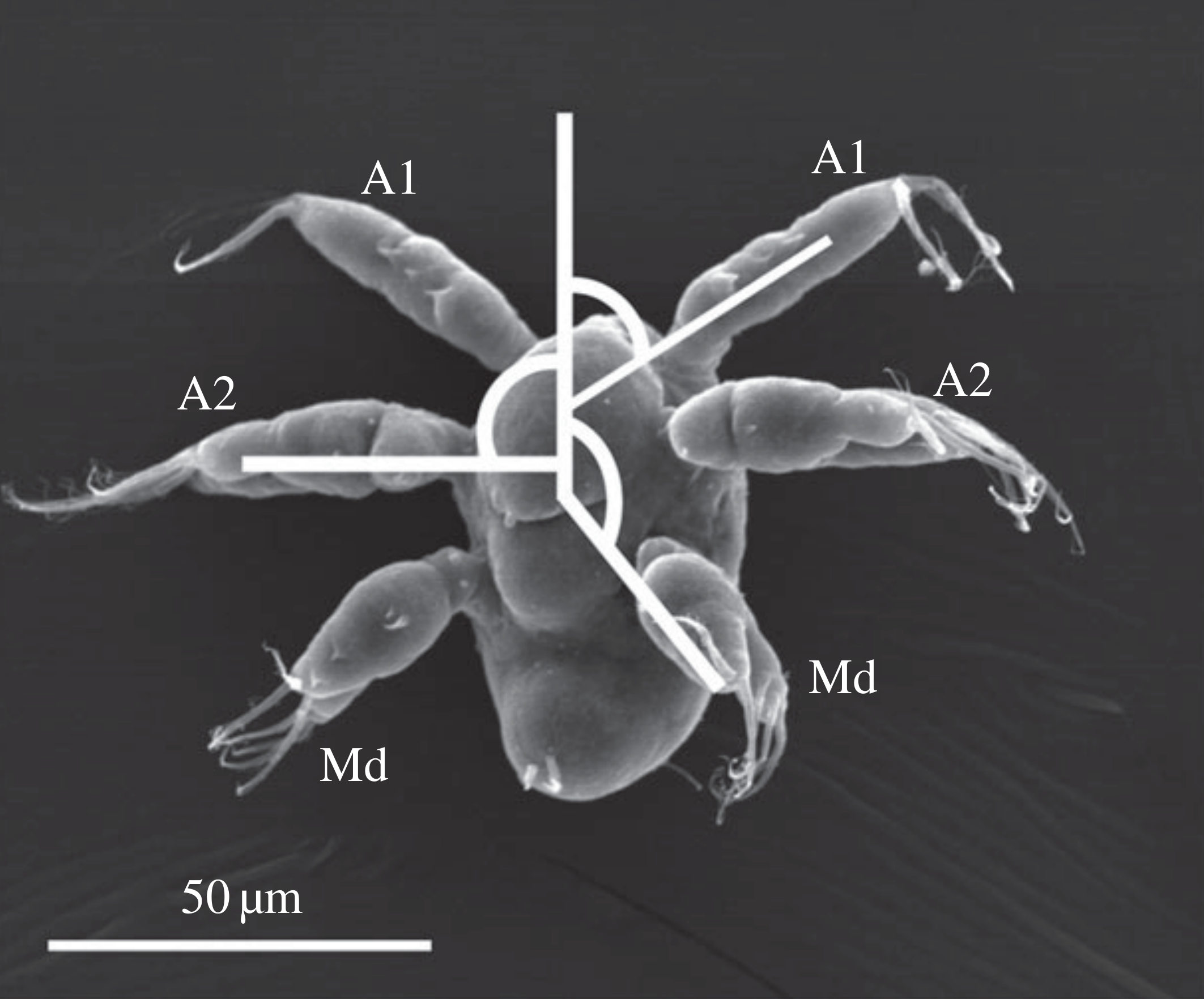}
\caption{Scanning electron micrograph of a larval copepod, showing three pairs of appendages: first antenna (A1), second antenna (A2) and mandible (Md). Image reproduced from \cite{optimal3} under Creative Commons Attribution 4.0 International License.
\label{fig: copepod photo}}
\end{figure}

Edward Mills Purcell's talk from 1976, Life at Low Reynolds Number (\cite{purcell}), first popularized the concept of swimming in an environment with low Reynolds number and was foundational in the study of microswimmers. He stated the scallop theorem, which says that complete reciprocal motion cannot produce any displacement when swimming at low Reynolds number. 
Here a reciprocal motion involves two sequences of motion where the second motion, called the recovery stroke, is the reverse of the first motion, which is called the power stroke. For humans swimming in water, the time it takes to do a stroke plays a role in the induced displacement due to the inertia terms in the Navier-Stokes equations.  So humans are able to swim forward by completing the second recovery stroke faster than the first.  But for microswimmers, this inertia is negligible, so any displacement induced by the first stroke is reversed by the second recovery stroke.  As a consequence, microorganisms must move in other ways, like utilizing a flagella or moving pairs of legs in an asynchronous manner. For example, for copepods that are bilaterally symmetrical, the symmetric pairs of legs move simultaneously,  but adjacent legs move out of sync in order to produce a net displacement. These observations are important as they could potentially be used to design microscopic robots that move in similar ways. One application involves using bacteria based nanoswimmers to transport drugs from a loading point to a destination such as cancer cells (\cite{nanoswimmer}). %These bacteria based nanoswimmers could potentially be used for drug transportation to cancer cells.

%\begin{figure}[H]
%\centering
%\includegraphics[width=0.5\textwidth]{3paircopepod.png}
%\caption{A nauplius of a type of paracalanid copepod %(\cite{rotational}).{\color{red} we cannot use this figure}}
%\label{fig: copepod photo}
%\end{figure}

The main microorganisms of focus in this paper are copepods. Most types of copepod are only able to move in a motion called swimming-by-jumping (\cite{kinematics}). This motion is similar to the one described in the previous paragraph: it involves moving symmetric pairs of legs in a way such that adjacent pairs move asynchronously. In other words, the power stroke and recovery stroke alternate among each of the pairs of legs. One way to model this is by moving each symmetric pair of legs in a reciprocal motion while introducing a phase lag between each pair of legs (\cite{takagi}). 
%The models considered in this paper are based on these copepods. It is a simplified model with long stiff legs and bodies of small radius. 
%Although copepods move in three-dimensions, we view the model as moving in two-dimensions for simplicity. 
%We also restrict the model's motion in the plane to be a series of concatenations of translational motions and rotational motions. 
As in such models, here we restrict the copepod's motion to a plane for simplicity, despite the fact that the actual animals live and move in three dimensions.

Here we model the copepod as a slender body in Stokes flow as in \cite{takagi,optimal1, optimal2,optimal3, devine,  book}. Other models of micro-swimmers capable of rotation appear in \cite{Dreyfus, Rizvi, Jalali}. 
It is important to note that real living copepods do indeed perform rotations to both evade predators (\cite{Robinson}) and to capture prey (\cite{Bruno}).
In \cite{rotational}, the authors analyze such rotational maneuvers (yaw, pitch and roll) via high-speed video observations of copepod larvae.

One dimensional translational motion for the copepod model has been well studied. 
It is possible to achieve positive displacement along an axis using as few as two pairs of legs moving in a reciprocal motion (\cite{takagi}). 
Methods from sub-Riemannian geometry and Hamiltonian dynamics have been used to find efficient optimal strokes in the translational case; efficiency was defined as the ratio between the displacement resulting from a stroke and the length of the stroke.  Numerical methods were used to determine the optimal strokes maximizing this efficiency (\cite{optimal1, optimal2,optimal3}).   Here the term stroke refers to a periodic motion of the legs.

%\newpage
% The notion of turning efficiency is also introduced as the ratio between the work required for an outside force to turn the swimmer a certain amount and the work the swimmer needs to turn itself that same amount. 
%Controur graphs are also plotted for these two values.

%A question of optimality arises when one considers how these copepods try to move from one point in the plane to another. As they are abundant creatures that have survived for a long time, they have developed ways to efficiently use their energy to hunt for food and evade predators. With the precedence of optimizing the energy, optimal strokes for the motion along an axis have been found and analyzed by use of the Maximum principle in the context of sub-Riemannian geometry \cite{optimal1,optimal2,optimal3}. We can apply a similar technique of using the Maximum principle to the same problem of optimizing the turning efficiency. Then if we concatenate the two motions by first rotating so that the copepod’s orientation point towards the terminal point and then applying the optimal strokes in moving along a line, we have one possibility of an energy optimizing motion from one point in the plane to another.

Here we generalize this prior work by analyzing planar motions. To produce orientation changes we need to break the symmetry of the pair of legs. A first attempt was made in \cite{devine} by looking at three independent legs oscillating sinusoidally; here we generalize that approach to include all strokes but using two legs. % {\color{red} I rewrote this part above, please check it is ok}.  
We find that rotation by strokes is indeed possible with only two legs.   We also show that the two-legged system is controllable, although the difficulty in steering locally depends on the initial state; in other words, the system possesses singularities, which we classify.  Taking the mechanical energy expended as our cost function, we develop the two-legged copepod movement as an optimal control problem and apply the Pontryagin maximum principle (\cite{Pontryagin}) to study both the normal and abnormal extremals.  We partially characterize the abnormal extremals, and provide some explicit examples.  Finally, we utilize the optimal control software \texttt{Bocop} to simulate normal extremals (\cite{bocop1}). 
%{\color{red} add a reference for Bocop, I also think it need capital letter and you need to explain it is an optimization software. Ok now I see you do explain below so maybe just put the reference here}.  
Among our simulations we find copepod motions which produce net rotation without net displacement,  we characterize the optimal %{\color{red} with respect to what?} 
motions which produce rotation with no conditions on displacement,  and we discover paths in the $xy$-plane which appear to be Euler elastica (\cite{elastica}).
%{\color{red} add a reference for this, George you probably have a suggestion}.

%{\color{blue}  TODO: shorten this and clean it up.}
 
%%%%%%%%%%%%%%%%%%%%%%%%%%%%%%%%%%%%%%%%%%
\section{Methods}\label{sec: methods}
%\textit{Roughly give model, eqns of motion, constraints, state optimal control problem}

We consider a simplified copepod microswimmer in a low Reynolds number environment.  The idealized copepod consists of stiff slender legs and a body of negligible radius in comparison to the length of its legs.  In this section we will develop our mathematical model, derive the equations of motion, describe the copepod's motion as an optimal control problem, and develop the appropriate version of the maximum principle.

\subsection{Model}

We assume the copepod moves in a plane and possesses $6$ independently moving legs, three on each side of the body.   The position of the copepod at time $t$ can be described by the vector $ \left(x(t),  y(t), \phi(t) \right)^T$, where $x$ and $y$ represent the usual Cartesian coordinates on the plane and $\phi$ represents the orientation of the copepod with respect to the positive $x$-axis. Let $\theta_i$ denote the angle between the copepod's orientation and the $i^{th}$ leg, and let $\alpha_i=\theta_i+\phi$ denote the angle between the $i^{th}$ leg and the positive $x$-axis.   See Figure \ref{fig: model n=6} for an illustration. % {\color{red} we should have an illustration with 2 legs since it is what we study. And you would keep the orientation behind the body as well so we clearly show there is one leg on each side. Or you do one picture with more, maybe 5 and then next to it a picture with 2 since it is what we are considering}.

\begin{figure}%[H]
\includegraphics[width=10.5 cm]{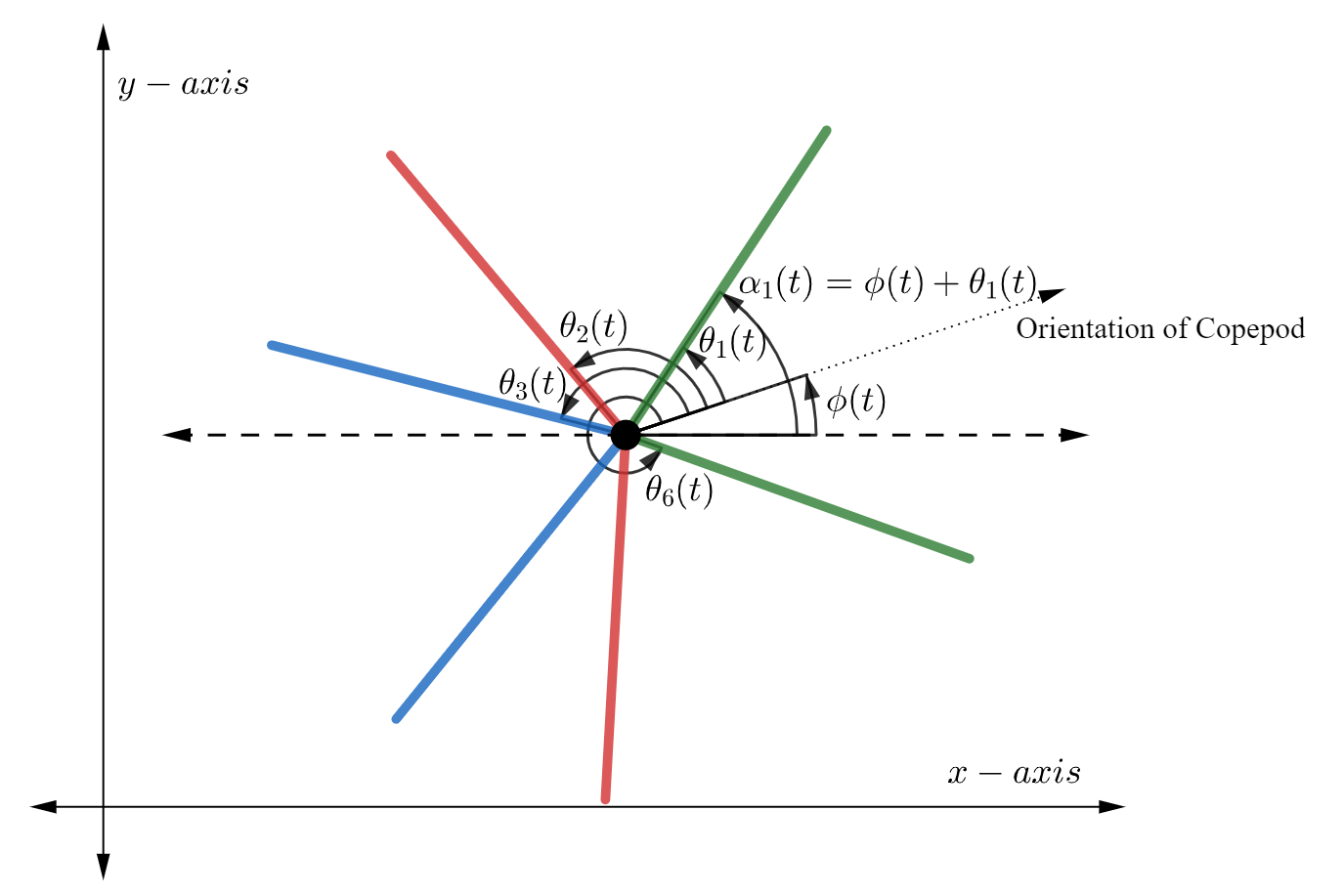}
\caption{The rotating copepod with $6$ legs. Note that angles $\theta_i$ are associated to the body frame while the angles $\alpha_i$ are associated to the inertial frame .  \label{fig: model n=6}}
\end{figure}
We denote the state of the copepod at time $t$ by the vector
$${q}(t)= \left( x(t),  y(t), \phi(t), \theta_1(t), \dots \theta_6(t) \right)^T,$$
while the position and orientation coordinates alone will be written as  $\hat{{q}}= \left( x(t),  y(t), \phi(t) \right)^T.$  Thus our configuration space is ostensibly $\mathbb R^{9}$, however in order to prevent the legs from passing each other,  we impose the constraint 
\begin{equation}\label{eq: constraint}
0 \leq \theta_1 \leq \theta_2 \leq \theta_3 
\leq \pi \leq \theta_4 \leq\theta_5 \leq\theta_6 \leq  2\pi.
\end{equation}

%\end{paracol}
%\begin{figure}[H]	
%\widefigure
%\includegraphics[width=15cm]{3legsym.png}
%\caption{The rotating copepod for $n=3$. Note that the $\theta_i$'s are with respect to the body frame and the $\alpha_i$'s are with respect to the inertial frame. \label{fig: model n=3}}
%\end{figure}  
%\begin{paracol}{2}
%\linenumbers
%\switchcolumn

For the rest of this paper we will focus on a simplified copepod with two independent legs, one on each side of the body. See Figure \ref{fig: model n=2}. This simplification allows us to conduct a mathematical analysis and is justified by assuming the three legs on each side of the body are collapsed into one stronger leg. As will be seen in Section \ref{sec: discussion}, even with this simplification we obtain swimming motions reflecting actual laboratory observations. 

\begin{figure}%[H]
\includegraphics[width=10.9 cm]{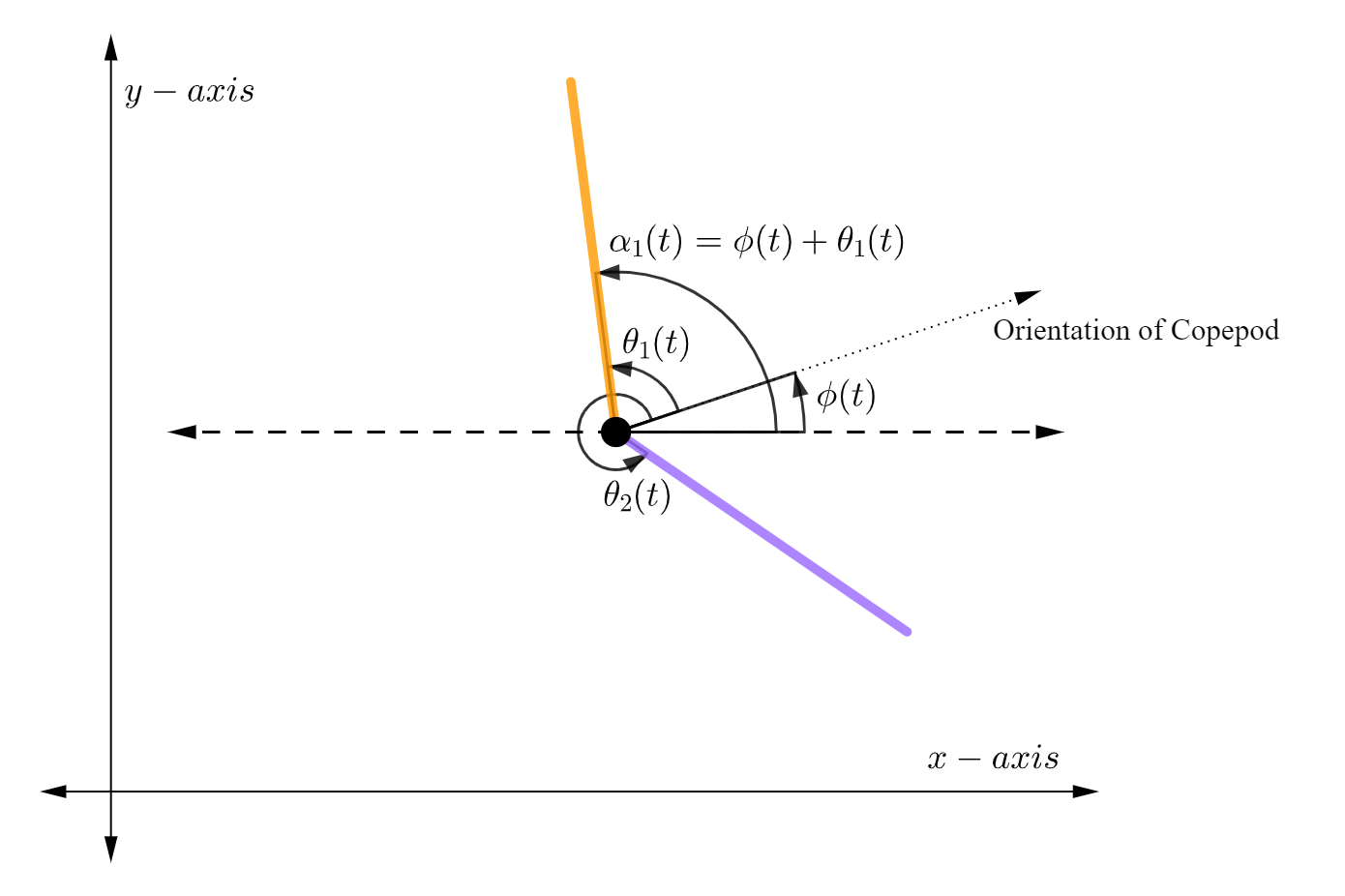}
\caption{The rotating copepod with $2$ legs.  \label{fig: model n=2}}
\end{figure}

%\begin{figure}[H]
%\includegraphics[width=10.5 cm]{3legsym.png}
%\caption{This is a figure. Schemes follow the same formatting. If there are multiple panels, they should be listed as: (\textbf{a}) Description of what is contained in the first panel. (\textbf{b}) Description of what is contained in the second panel. Figures should be placed in the main text near to the first time they are cited. A caption on a single line should be centered.\label{fig1}}
%\end{figure}   

In most of Section \ref{sec: results}, including all of Section \ref{subsec: abnormal}, we use standard techniques from optimal control (\cite{Pontryagin}) and sub-Riemannian geometry (\cite{SR book}).  In Section \ref{subsec: normal}, however, we utilize the optimal control software \texttt{Bocop}.  As stated in \cite{bocop1}, 
this software approximates our optimal control problem by a finite dimensional optimization problem using the direct transcription approach to time discretization. The resulting nonlinear programming problem is solved using the software package \texttt{Ipopt}, using sparse exact derivatives computed by \texttt{ADOL-C}.

\subsection{Equations of Motion}
%{\color{blue} In this section and the next we still use $n$ for the number of legs.  We can change it to 2 or 6, or just say that this model works for $n$ legs.  The latter is probably easiest but can discuss.}{\color{red} yes I agree we can just say it works for $n$.}
We first develop the equations of motion for $n$ legs, then specify to the case $n=2$.
Our equations consist of a system of differential equations of the form ${M} \dot{\hat{{q}}} =  K$.  % (see \cite{devine}). 
The equations of motion for a copepod moving in two dimensions are derived in \cite{devine}, which focuses on legs moving in an oscillatory motion: $\theta_i(t)=a \cos(t+k_i)+\beta_i$.
Parameters are constrained to ensure that adjacent legs never overlap but possess a phase lag. 
The author shows that no net rotation is possible with such a motion for two legs, thus most of the analysis concerns the case of three legs.
%he focus of this work was to study strokes producing a net rotation and displacement. 
%Part of his thesis work involves using these equations to obtain several numerical results from a three-legged model whose legs move in an oscillatory motion: $\theta_i(t)=a \cos(t+k_i)+\beta_i$ for $i=1,2,3$. 
For numerical simulations, the values of $a,k_1,\beta_1,\beta_2,$ and $\beta_3$ are fixed and the total change in orientation and displacement is computed for varying values of $k_2$ and $k_3$. 
%Two contour graphs are then plotted in the $k_2,k_3$ plane, one depicting the change in orientation and the other depicting the displacement. 
The change in displacement and orientation is maximized when $(k_2,k_3)=(\frac{2 \pi}{3},\frac{4 \pi}{3})$ and $(k_2,k_3)=(\frac{4 \pi}{3},\frac{2 \pi}{3})$. 
In addition, the total work done by the microswimmer is calculated and a notion of turning efficiency is introduced.
%{\color{blue} Monique, I believe you deleted this paragraph (except first sentence) and I accidentally re-introduced it.  Correct?}{\color{red} it is fine here, it was just that it was in the introduction before. It has its place we should keep it.}

The system is derived from slender-body Stokes flow, using the fact that, at low Reynolds number, inertial forces are negligible and the Navier-Stokes equations can  be linearized.   Here %{\color{red} do we want write that the sum is on the number of legs?} 
\begin{equation} \label{eq: K}
     K=\sum_{i=1}^n\dot\theta_i \left (\begin{array}{ccc}
 \sin\alpha_i \\
-\cos\alpha_i \\
-\frac{2}{3}
\end{array}\right )
\end{equation}
and $ M$ is the resistance matrix given by
\begin{equation} \label{eq: M}
     M=\left (\begin{array}{ccc}
 \sum (1+\sin^2\alpha_i) & -\sum(\sin\alpha_i\cos\alpha_i) & - \sum\sin\alpha_i\\
-\sum(\sin\alpha_i\cos\alpha_i) &  \sum(1+\cos^2\alpha_i) &  \sum\cos\alpha_i\\
- \sum\sin\alpha_i &  \sum\cos\alpha_i\ & 2
\end{array}\right ).
\end{equation}
By computing the mobility matrix $ M ^{-1}$ (well defined since $M$ is symmetric and positive definite) we obtain the equations of motion
\begin{equation}
\dot{\hat{{q}}}=  M ^{-1}  K.
\end{equation}
% which we can use to solve for $\dot x$, $\dot y$, and $\dot \phi$. 

\subsection{Optimal Control Framework}
We now consider this system from the control theoretic point of view, where the angular velocities of the leg are taken as controls.  That is, we set $u_i=\dot \theta_i$, and assume these are measurable functions of time.
Now let 
\begin{equation} \label{Ki}
    {K}_i=\left (\begin{array}{ccc}
 \sin\alpha_i \\
-\cos\alpha_i \\
-\frac{2}{3}
\end{array}\right).
\end{equation}
Then our control vector fields are
\begin{equation}
{F}_i = \left( {M}^{-1}  {K}_i, 0, ..., 0, 1, 0, ..., 0 \right)^T,
\end{equation}
 where the $1$ appears in the $i^{th}$ entry after the ${M}^{-1} {K}_i$ entries. 
Then the copepod motion is described by the driftless affine control system
\begin{equation} \label{eq: Control System}
    {\dot q}(t)=\sum_{i=1}^n u_i(t) {F}_i({q}(t))
\end{equation} 
where $n$ refers to the number of legs. In this work we assume no bounds on the control.  In reality, of course, there are limits to how quickly an actual copepod can move its legs.  But some of this issue is mitigated by the fact that we will be minimizing some form of energy; see Equation (\ref{eq: translationalenergy2}) below.
%({\color{red} should we say something more? Clearly they cannot move their legs as fast as they can.. but since we are minimizing some form of energy they will not want to do that anyway}). 
We will see in the next section that, when $n=2$, this is a controllable system; that is, it is possible to find controls steering the copepod from any given initial state $q_{\text{initial}}$ to any given terminal state $q_{\text{final}}$. %{\color{red} we are doing that for 2 legs only correct? so maybe it should come later? or actually if we can do with two legs with can do with any?}.

he rest of this paper will concern the special case of the copepod with two legs, so here we record these equations of motion explicitly.  
When $n=2$, we have
\begin{equation*}
\resizebox{1\hsize}{!}{$
{M}=\left (\begin{array}{ccc}
 2+\sin^2\alpha_1+\sin^2\alpha_2 & -\sin\alpha_1\cos\alpha_1 - \sin\alpha_2\cos\alpha_2 & - \sin\alpha_1- \sin\alpha_2\\
-\sin\alpha_1\cos\alpha_1 - \sin\alpha_2\cos\alpha_2 &  2+\cos^2\alpha_1+\cos^2\alpha_2 &  \cos\alpha_1+\cos\alpha_2\\
- \sin\alpha_1- \sin\alpha_2 &  \cos\alpha_1+\cos\alpha_2 & 2
\end{array}\right )$}
\end{equation*}
and 
$$
{K}=\dot\theta_1 \left (\begin{array}{ccc}
 \sin\alpha_1 \\
-\cos\alpha_1 \\
-\frac{2}{3}
\end{array}\right ) + \dot\theta_2 \left (\begin{array}{ccc}
 \sin\alpha_2 \\
-\cos\alpha_2 \\
-\frac{2}{3}
\end{array}\right )=\dot \theta_1 {K}_1 + \dot \theta_2 {K}_2.
$$ 
Our control system then takes the form 
\begin{equation} \label{eq: Control System 2 legs}
    {\dot q}(t) = u_1(t){F}_1(q(t))+u_2(t){F}_2(q(t).
\end{equation} 
where ${q}=(x, y, \phi, \theta_1, \theta_2)^t,$ our controls are $u_1=\dot \theta_1$ and $u_2=\dot \theta_2$, and the control vector fields are
$${F}_1= \left (\begin{array}{ccc}
{M}^{-1}{K}_1 \\
 1 \\
 0
\end{array}\right )
\qquad \text{and} \qquad
{F}_2= \left (\begin{array}{ccc}
{M}^{-1}{K}_2 \\
 0 \\
 1
\end{array}\right ).
$$
Straightforward calculations give
\begin{equation} \label{eq: 2legF1}
{F}_1= \left (\begin{array}{ccc}
-\frac{\sin(\theta_1-2\theta_2-\phi)-\sin(2\theta_1-\theta_2+\phi)+17\sin(\theta_1+\phi)-7\sin(\theta_2+\phi)}{24(\cos(\theta_1-\theta_2)-3)}\\
-\frac{\cos(\theta_1-2\theta_2-\phi)+\cos(2\theta_1-\theta_2+\phi)-17\cos(\theta_1+\phi)+7\cos(\theta_2+\phi)}{24(\cos(\theta_1-\theta_2)-3)}\\
\frac{1}{12}(\cos(\theta_1-\theta_2)-3)\\
 1 \\
 0
\end{array}\right ),
\end{equation}

\begin{equation} \label{eq: 2legF2}
{F}_2= \left (\begin{array}{ccc}
-\frac{\sin(\theta_1-2\theta_2-\phi)-\sin(2\theta_1-\theta_2+\phi)+17\sin(\theta_2+\phi)-7\sin(\theta_1+\phi)}{24(\cos(\theta_1-\theta_2)-3)}\\
-\frac{\cos(\theta_1-2\theta_2-\phi)+\cos(2\theta_1-\theta_2+\phi)-17\cos(\theta_2+\phi)+7\cos(\theta_1+\phi)}{24(\cos(\theta_1-\theta_2)-3)}\\
\frac{1}{12}(\cos(\theta_1-\theta_2)-3)\\
 0 \\
 1
\end{array}\right ).
\end{equation}

In the sequel we let $D$ denote the distribution spanned by $F_1$ and $F_2$.
Note that we can obtain ${F}_2$ from ${F}_1$ by switching the roles of $\theta_1$ and $\theta_2$.  Moreover, system (\ref{eq: Control System 2 legs}) is time-reversible; indeed the transformation $t\mapsto 2\pi-t, u_i(t)\mapsto -u_i(2\pi-t)$ sends $q(t)$ to $q(2\pi-t)$. This is a general consequence of swimming at low Reynolds number. The system is also invariant under rigid body transformations. It is obvious for translation as the $F_i$ do not depend on $x$ or $y$, but it can also be verified that $\phi\mapsto \phi+\tau$ is invariant under $x\mapsto x\cos\tau-y\sin\tau, y\mapsto x\sin\tau+y\cos\tau$.

In this paper we suppose that the copepod seeks to minimize the mechanical energy expended when moving from one position to another.  In \cite{optimal3}, the authors consider a two-legged copepod moving along an axis without rotation.  They describe a realistic but complicated mechanical energy functional, but show that the resulting optimal trajectories are qualitatively similar to those obtained when using the simplified energy
\begin{equation} \label{eq: translationalenergy2}
    E(u_1,u_2)=\int_{t_0}^{t_f}
\left(u_1^2+u_2^2\right) \, dt.
\end{equation}
Here $t_0$ is a fixed initial time, while $t_f$ is associated to the control $u$ in the following manner.  Choose some terminal boundary manifold $M_1 \subseteq \mathbb R^5$, which is closed, and define the target set by $M=[t_0, \infty) \times M_1$.  Then $t_f$ is the smallest time such that $(t_f, q(t_f))\in M$, where $q(t)$ is the state trajectory associated to the control $u(t)$.  

We therefore choose our cost function to be the energy $E$ given in (\ref{eq: translationalenergy2}) corresponding to the orthonormal inner product for the two control vector fields, yielding the following optimal control formulation.  Provided certain boundary conditions made explicit below, we seek solutions to the dynamical system (\ref{eq: Control System 2 legs}) which minimize the cost (\ref{eq: translationalenergy2}). Note that this is a sub-Riemannian problem associated to the flat metric (\cite{book2}). 

\subsection{Maximum Principle}
The Pontryagin maximum principle provides necessary conditions for a solution to be optimal. The general statement can be found in the literature \cite{book2, SR geo, Liberzon, Pontryagin}, and we here only state it for our application. 

Consider the optimal control problem stated in the previous section, defined by the dynamics (\ref{eq: Control System 2 legs}) and cost (\ref{eq: translationalenergy2}). 
Let $j, k \leq 5$ and consider the initial and terminal boundary manifolds
\begin{align*}
    M_0 &= \{q \in \mathbb R^5 \colon g_1(q)=g_2(q)=\cdots=g_j(q)=0\}\\
    M_1 &= \{q \in \mathbb R^5 \colon h_1(q)=h_2(q)=\cdots=h_k(q)=0\}.
\end{align*}
Define the Hamiltonian by 
\begin{equation}\label{eq: hamiltonian}
    H(q, u, p, p_0) = \langle p, u_1F_1 + u_2F_2 \rangle + p_0(u_1^2 + u_2^2).
\end{equation}

\begin{theorem}[Maximum Principle]\label{thm: max principle}
Let $u^* \colon [t_0, t_f] \to \mathbb R^2$ be an optimal control and let $q^* \colon [t_0, t_f] \to \mathbb R^5$ be the corresponding optimal state trajectory.  Then there exists a function $p^* \colon [t_0, t_f] \to \mathbb R^5$ and a constant $p_0^* \leq 0$ such that $(p_0^*,\, p^*(t)) \neq (0,0)$ for all $t \in [t_0, t_f]$ and having the following properties:
\begin{enumerate}
    \item $q^*$ and $p^*$ satisfy Hamilton's equations for the Hamiltonian (\ref{eq: hamiltonian}) with boundary conditions $q^*(t_0)\in M_0$ and $q^*(t_f)\in M_1$.
    \item $H(q^*(t), u^*(t), p^*(t), p^*_0) \geq H(q^*(t), u, p^*(t), p^*_0)$ for all $t\in [t_0, t_f]$ and $u \in \mathbb R^2$.
    \item $H(q^*(t), u^*(t), p^*(t), p^*_0)=0$ for all $t\in [t_0, t_f]$.
    \item The vector $p^*(0)$ is orthogonal to the tangent space $T_{q^*(t_0)}M_0$ and the vector $p^*(t_f)$ is orthogonal to the tangent space $T_{q^*(t_f)}M_1$.
\end{enumerate}
\end{theorem}

Note that this version of the maximum principle closely follows Section 4.1.2 of \cite{Liberzon}, with the addition of the initial boundary conditions $M_0$ and the associated transversality conditions.  It important, however, to recognize that Theorem \ref{thm: max principle} pertains to the free-time problem.  In our simulations in Section \ref{subsec: abnormal} we fix the time $[t_0, t_f]=[0, 2\pi]$, which requires a minor modification of the maximum principle.  As described in Section 4.3.1 of \cite{Liberzon}, one simply introduces an extra state variable to represent time, $q_6=t$, and includes the fixed terminal time in the terminal manifold $M_1$.  Then the free-time Theorem \ref{thm: max principle} applies with only the following modification: $H|_*=-p_6^*$, which is constant.  In the sequel we will refer to a trajectory satisfying the conclusions of the maximum principle as an extremal. 

%{\color{blue} I lean towards brevity with references for details.  We have to be careful with endpoints: in the simulations we typically specify some initial and final coordinates, but never the entire initial or final state.  Not sure how to correctly state the PMP for this case.  Maybe Monique you can help write this part?}

%%%%%%%%%%%%%%%%%%%%%%%%%%%%%%%%%%%%%%%%%%
\section{Results}\label{sec: results}

Most of the results in this section concern the two-legged copepod.  Unless explicitly stated otherwise, we assume $n=2$.  The motivation for this choice is contained in the next theorem, which says that one leg is insufficient for producing rotation via periodic strokes, but two legs are sufficient.  As in \cite{book}, we define a \emph{stroke} to be a periodic deformation of the swimmer's body.    That is, a stroke of period $T$ is any path in configuration space satisfying $\theta_i(0)=\theta_i(T)$ for all $i=1, \dots n$.  
For simulations and examples we also impose the following realistic constraint forcing each leg to stay on one side of the copepod's body:
%, which is slightly stronger than (\ref{eq: constraint}):
\begin{equation}\label{eq: constraint 2 legs}
0 \leq \theta_1 \leq \pi \leq \theta_2 \leq 2\pi.
\end{equation}
However, most of the mathematical analysis in this section is valid for the configuration space $\mathbb R^2\times (S^1)^3$.
Note that  \cite{devine} proves that a two-legged copepod is incapable of producing net rotation via the specific oscillating strokes considered in their work.

\begin{theorem} A one-legged copepod moving in strokes can neither produce a net rotation nor net displacement.   A two-legged copepod moving in strokes can produce net rotation.
\end{theorem}

\begin{proof}
For a copepod with one leg, we compute 
$$
 M=\left (\begin{array}{ccc}
 1+\sin^2\alpha & -\sin\alpha\cos\alpha & - \sin\alpha\\
-\sin\alpha\cos\alpha &  1+\cos^2\alpha &  \cos\alpha\\
- \sin\alpha &  \cos\alpha\ & 2
\end{array}\right )
\qquad \text{and} \qquad 
 K=\dot\theta \left (\begin{array}{ccc}
 \sin\alpha \\
-\cos\alpha \\
-\frac{2}{3}
\end{array}\right ).
$$ 
Then the equation $\dot{\hat{{q}}} =  M^{-1}  K$ simplifies to
$$
\left (\begin{array}{ccc}
 \dot x \\
 \dot y\\
 \dot \phi
\end{array}\right ) = \dot\theta \left (\begin{array}{ccc}
  \frac{4}{9} \sin\alpha  \\
- \frac{4}{9} \cos\alpha  \\
- \frac{1}{9} 
\end{array}\right ).
$$
Suppose the copepod moves in strokes of period $T$, so $\theta(0)=\theta(T)$.  Without loss of generality assume $\phi(0)=0.$ Then $\phi(t)=-\frac{1}{9} \theta(t)+\frac{1}{9} \theta(0)$. Because $\theta(t)$ is periodic, $\phi(t)$ is also periodic and thus no net rotation is produced after one period. 

Now we also have that $\alpha(t)=\theta(t)+\phi(t)=\frac{8}{9}\theta(t)+\frac{1}{9}\theta(0)$. To calculate the total $x$ displacement over the period we compute the integral
$$\int_0^T \dot x(t) \, dt=\int_0^T \frac{4}{9}\sin\left(\frac{8}{9}\theta(t)+\frac{1}{9}\theta(0)\right)\dot \theta(t) \, dt .$$
Computing this integral gives us
$$x(T)-x(0)=\left. -\frac{1}{2}\left( \cos\left(  \frac{8}{9}\theta(t)+\frac{1}{9} \theta(0) \right)  \right) \right\vert _0^T,$$
which equals zero again since $\theta$ is periodic.
Thus the displacement in the $x$ direction is 0; a similar argument gives the result for $y$.

 The proof of the second statement in the theorem is contained in the following example.
 \end{proof}
The next example demonstrates that two legs moving in strokes can produce a net change in displacement and orientation.

\begin{example}\label{ex: triangle}
Consider the two-legged copepod with initial configuration
$${q}(0)=(x(0), y(0), \phi(0), \theta_1(0), \theta_2(0))^T=(0, 0, 0, 0, \pi)^T.$$
We consider the following motion. Both legs rotate $\pi$ radians counter-clockwise in $\pi$ time; then one at a time, each leg moves $\pi$ radians clockwise in $\frac{\pi}{2}$ time.  
See Figure \ref{fig: triangle}.  Explicitly, we have 
$$
\theta_1(t) = \begin{cases} 
          t & 0\leq t< \pi \\
          -2t+3\pi & \pi\leq t < \frac{3 \pi}{2} \\
          0 & \frac{3 \pi}{2} \leq t < 2\pi
       \end{cases} ,
       \indent 
       \theta_2(t) = \begin{cases} 
          t+\pi & 0\leq t< \pi \\
          2\pi & \pi\leq t < \frac{3 \pi}{2} \\
          -2t+5 \pi & \frac{3 \pi}{2} \leq t < 2\pi.
       \end{cases}
$$
We can explicitly solve the equations of motion for the orientation over time,
$$
\phi(t) = \begin{cases} 
          -\frac{2}{3}t & 0\leq t < \pi \\
          \frac{1}{12}\sin\left( 2t\right) +\frac{1}{2}t-\frac{7\pi}{6}& \pi\leq t \leq 2\pi, \\
       \end{cases}
$$
which implies a net rotation of
$\phi(2\pi) - \phi(0)=-\frac{\pi}{6} \neq 0.$  The orientation over time, along with the displacements over time obtained by numerical integration,  are shown in Figure \ref{fig: ex rot}.  
The net change in position, equal to the final position, is given by 
$$\hat{{q}}\left( 2\pi)=(x(2\pi),y(2\pi),\phi(2\pi) \right)^T=(0.0071,0.0019,-{\pi}/{6})^T.$$
\end{example}

\begin{figure}[H]
    \centering
\includegraphics[width=8 cm]{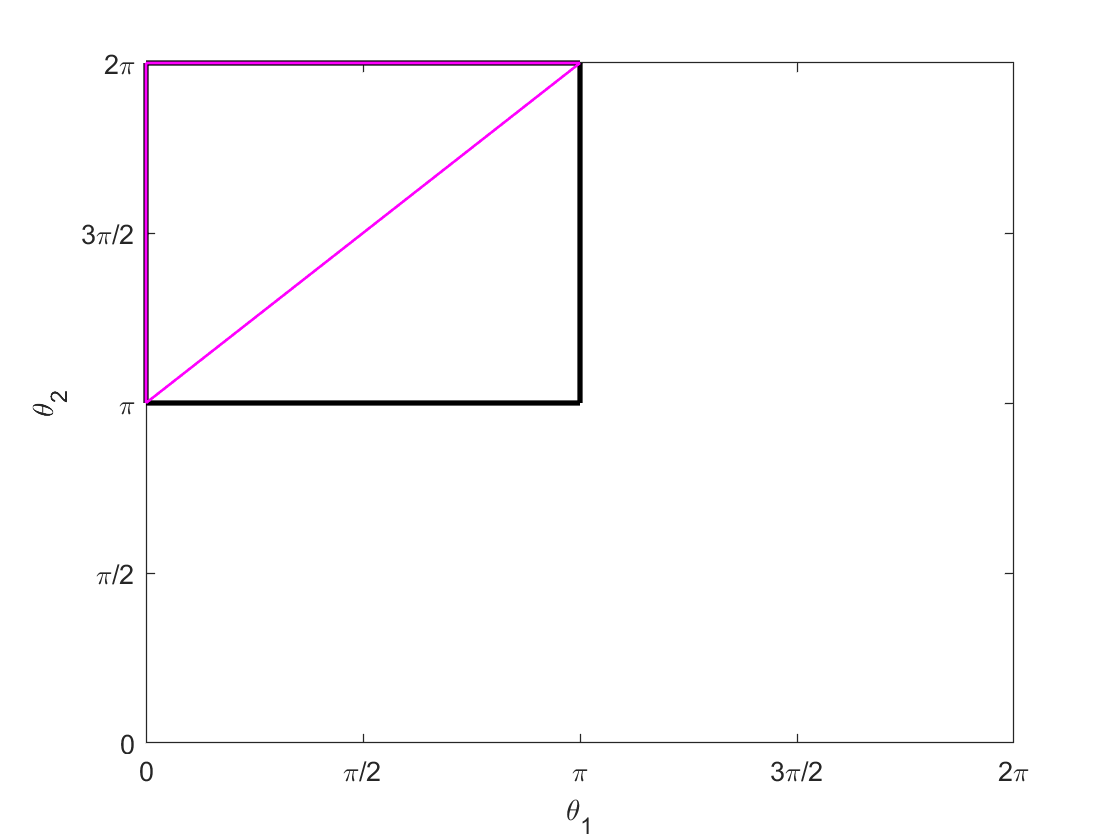}
    \caption{The motion described in Example \ref{ex: triangle} traces out the pink triangle in the $\theta_1\theta_2$-plane, counter-clockwise from the bottom left vertex. The black box represents the constraint (\ref{eq: constraint 2 legs}).\label{fig: triangle}}
\end{figure}

\begin{figure}[H]
   % \centering
   % \subfigure[]{\includegraphics[width=5 cm]{motion1triangle.png}} 
    \subfigure[]{\includegraphics[width=6.5 cm]{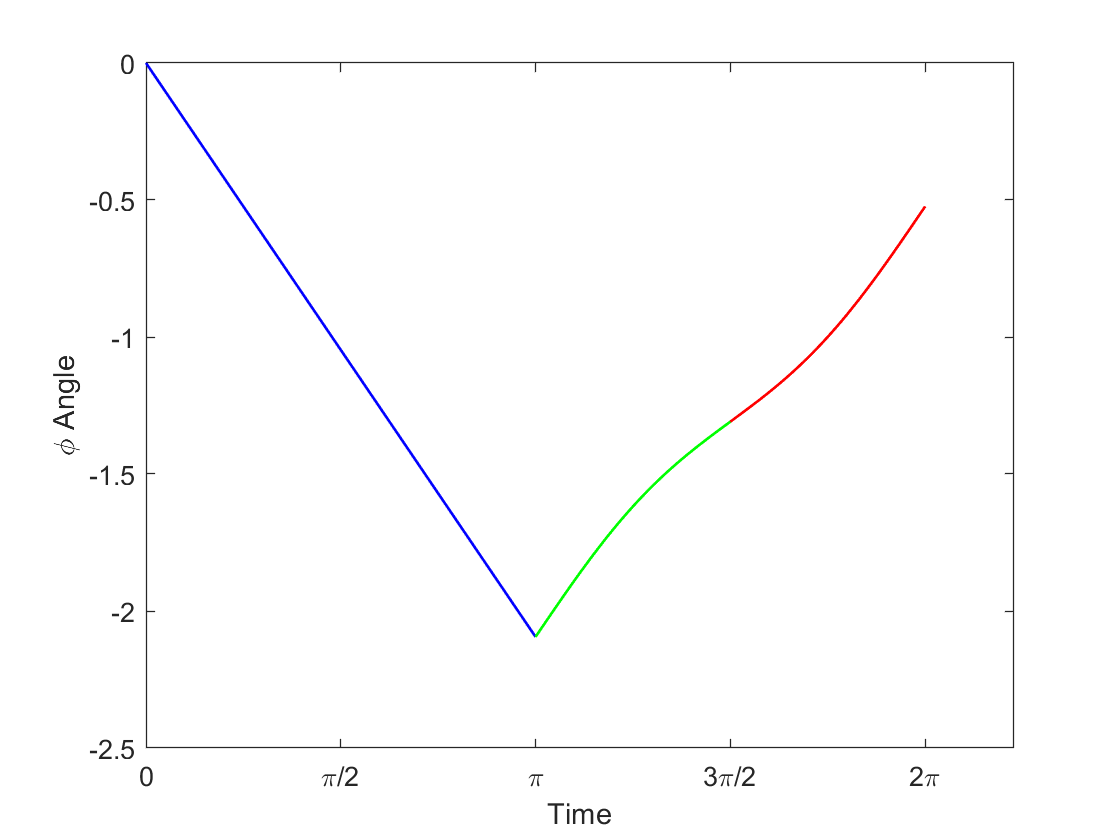}} 
    \subfigure[]{\includegraphics[width=6.5 cm]{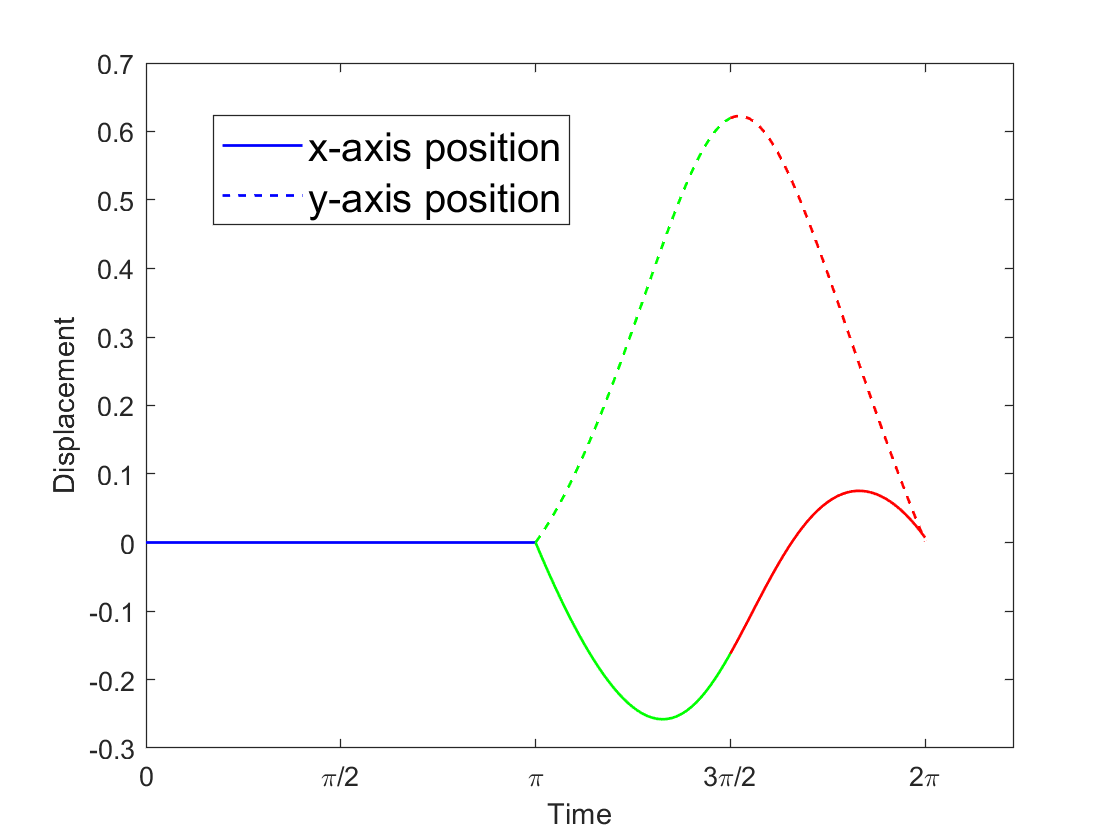}}
    \caption{Motion in Example \ref{ex: triangle}; the three colors correspond to the three legs of the triangle in Figure \ref{fig: triangle}.  (\textbf{a}) Orientation $\phi(t)$.  (\textbf{b}) Displacements $x(t)$ and $y(t)$.\label{fig: ex rot}}
\end{figure}

%\begin{figure}[H]
%\includegraphics[width=1.5 cm]{2leg2phidisp.png}
%    \caption{Orientation $\phi(t)$ in Example \ref{ex: triangle}; the three colors correspond to the three legs of the triangle in Figure \ref{fig: triangle}.\label{fig: ex rot phi}}
%\end{figure}

%\begin{figure}[H]
%\includegraphics[width=1.5 cm]{2leg2xydisp.png}
  %  \caption{Displacements $x(t)$ and $y(t)$ in Example \ref{ex: triangle}; the three colors correspond to the three legs of the triangle in Figure \ref{fig: triangle}.\label{fig: ex rot xy}}
%\end{figure}

Note that the energy (\ref{eq: translationalenergy2}) for the motion in Example \ref{ex: triangle} is equal to $6\pi$.  %, which is relatively large {\color{red} in comparison to what? how do we know it is relatively large?}.
While this motion is dynamically valid, it is likely not minimizing the cost.  See Section \ref{subsec: normal} for further discussion.  %Moreover, in Section \ref{section-controllability} we will see that for any initial $q_{\text{initial}}$ and final $q_{\text{final}}$ states for the copepod, including the positioning of the legs,  there are controls $u_1$ and $u_2$ that steer system (\ref{eq: Control System 2 legs}) from $q_{\text{initial}}$ to $q_{\text{final}}$. 

%{\color{red} Can we say what is the energy for this motion? You have to calculate it. I would like to explain that while we have a motion, it likely not minimizing the cost.} 

\subsection{Controllability}  
\label{section-controllability}

In this subsection we show that the two-legged copepod is indeed a controllable system.  That is, given any initial and final configuration, controls exist which steer the copepod from the initial to the final configuration.  The main tool here is the Chow-Rashevskii theorem; a formal statement, along with definitions of all the terminology in this subsection, can be found in \cite{SR book}.  The proofs here are essentially just calculations, which we performed using MATLAB and Mathematica.  

Our two control vector fields $F_1$ and $F_2$ are given by (\ref{eq: 2legF1}) and (\ref{eq: 2legF2}).  Denote their iterated Lie brackets (which are too complicated to display) by 
%{\color{red} those vectors are too complicated to show correct?}
\begin{equation*}
F_3 = [F_1, F_2], \quad F_4=[F_1, F_3],\quad F_5=[F_2, F_3]. 
\end{equation*}

Note that $F_1$ and $F_2$, and consequently their iterated brackets,  only depend on $\theta_1, \theta_2$, and $\phi$.    A computation shows that the five vector fields $F_1, F_2, F_3, F_4,$ and $F_5$ are linearly dependent if and only if 
\begin{equation}\label{eq: det1}
{\sin^4\left(\frac{\psi}{2}\right) \left( 25 \cos(2 \psi)+120 \cos(\psi)+ 79 \right)}=0
\end{equation}
where $\psi=\theta_1-\theta_2$.
Note that this is only a condition on two of our five variables.  Solving (\ref{eq: det1}) yields two sets of solutions:
$\psi=2\pi n$ and $\psi=2\pi n \pm 2\arctan(2) $ for $n \in \mathbb Z$.  We therefore have two sets of configurations which are singular for our distribution $D$:
\begin{align*}
S_1 &= \{ {q}\, |\, \theta_1-\theta_2=2\pi n \pm 2\arctan(2)\, \text{for some}\, n \in \mathbb Z\} \\
S_2 &= \{ {q} \,|\, \theta_1-\theta_2=2\pi n\, \text{for some}\, n \in \mathbb Z\}.
\end{align*}
These are illustrated in the $\theta_1\theta_2$-plane in Figure \ref{fig: sing sets}.

\begin{figure}[H]
\centering
\includegraphics[width=8.5 cm]{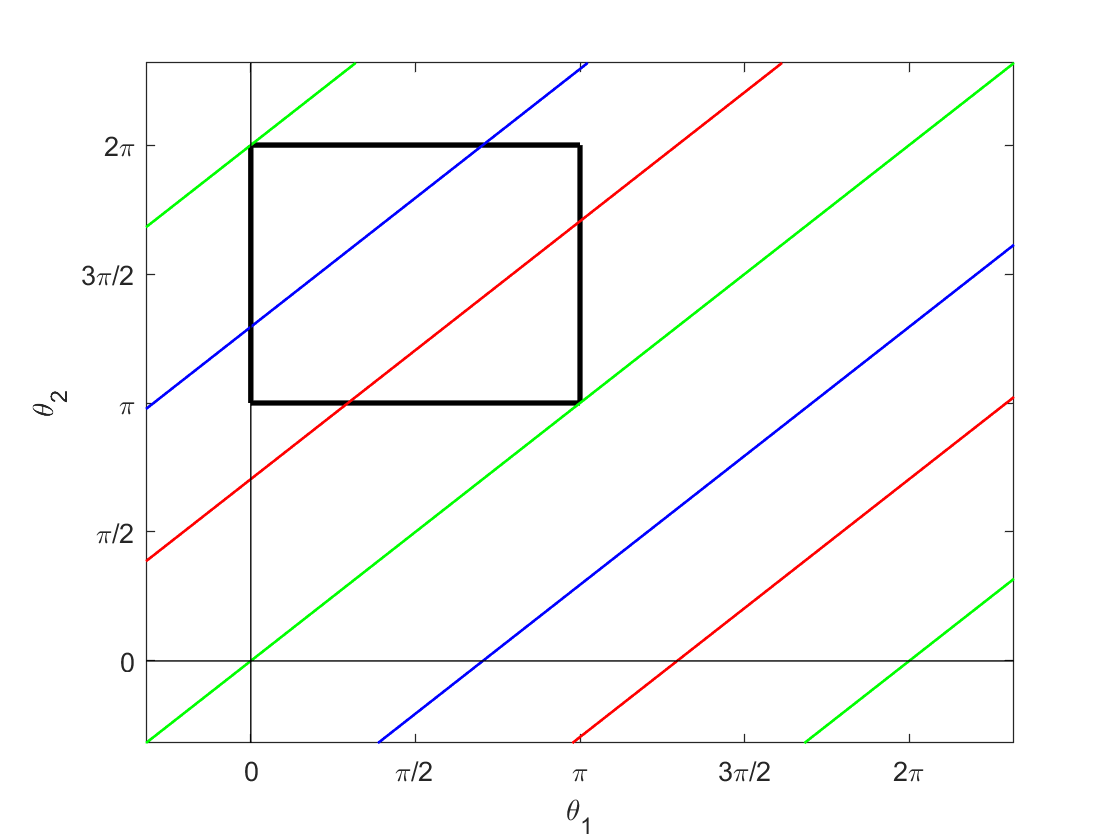}
    \caption{The singular sets for $D$.  The green lines are given by $\psi=2\pi n$, the red lines by $\psi=2\pi n - 2\arctan(2)$, and the blue lines by $2\pi n + 2\arctan(2)$, where $n\in \mathbb Z$ and $\psi=\theta_1-\theta_2$.  The black box again shows the constraint (\ref{eq: constraint 2 legs}).\label{fig: sing sets}}
\end{figure}

\begin{theorem}\label{thm: sing}
If ${q} \in S_1$ then the small growth vector at ${q}$ is (2, 3, 4, 5).  
If ${q} \in S_2$ then the small growth vector at ${q}$ is (2, 2, 3, 4, 5).    All other points are regular with small growth vector  (2, 3, 5).
\end{theorem}

\begin{proof}
The fact that the small growth vector is (2, 3, 5) at generic points is immediate from the fact that  $F_1, F_2, F_3, F_4,$ and $F_5$ are linearly independent there.  These points are regular since the singular sets are closed.  Points in the singular sets are analyzed by computations. Let
\begin{equation*}
F_6 = [F_1, F_4], \quad F_7=[F_1, F_5],\quad F_8=[F_2, F_4],  
\quad F_9 = [F_2, F_5], \quad F_{10}=[F_1, F_6].
\end{equation*}

For points in $S_1$, we find that $F_4=-F_5$, but $F_1, F_2, F_3, F_4,$ and $F_6$ are linearly independent.  For points in $S_2$, we find that $F_3=F_7=F_8=0$ and $F_4=-F_5$ and $F_6=-F_9$, but $F_1, F_2, F_4, F_6,$ and $F_{10}$ are linearly independent.
\end{proof}

\begin{Corollary}
The two-legged copepod system is controllable at all points.  At generic points the degree of non-holonomy is 3.  At points in $S_1$ the degree of non-holonomy is 4. At points in $S_2$ the degree of non-holonomy is 5.
\end{Corollary}

\begin{proof}
Controllability follows from the Chow-Rashevskii theorem, as the vector fields $F_1$ and $F_2$ Lie generate the tangent bundle at every point.  The degree of non-holonomy is simply the length of the small growth vector.
\end{proof}

\subsection{Abnormal extremals.} \label{subsec: abnormal}
%No vert or horiz line segments.  Different from transl.  Only abnormals in sing set produce no rotation or displacement.   Maybe compute abnormal curves as integral curves for vector field (should be one dim) which are both horizontal and Cauchy characteristic for $[D, D]$ (see Prop 5.7 in \cite{SR book}).
Abnormal extremals are intrinsic to the dynamics; they do not depend on the cost. It is well known that they play a very important role for the optimal synthesis (\cite{book2}). They correspond to imposing $p_0=0$ in the Hamiltonian (\ref{eq: hamiltonian}). It follows that, for our application, the abnormal Hamiltonian is 
$$ H_a(q, p, u)= u_1\langle  p, F_1(q) \rangle + u_2\langle  p, F_2(q) \rangle.
$$
According to the Pontryagin maximum principle, abnormal extremals are curves $(q(t), p(t))$ satisfying Hamilton's equations for $H_a$ as well as 
\begin{align}
    \langle p, F_1(q) \rangle &=0 \label{eq: abnormal1}\\
    \langle p, F_2(q) \rangle &=0. \label{eq: abnormal2}
\end{align}
Differentiating these equations leads to the additional requirements
\begin{align}
    \langle p, F_3(q) \rangle &=0\label{eq: abnormal3} \\
    u_1\langle p, F_4(q) \rangle + u_2\langle p, F_5(q) \rangle &=0.\label{eq: abnormal4}
\end{align}
The following results partially characterize the abnormal curves.
\begin{Proposition}\label{prop: explicit abnormals}
The horizontal lifts of the singular curves in Figure \ref{fig: sing sets} are projections of abnormal curves for the two-legged copepod. They are the integral curves for the vector field $F_1+F_2$ restricted to the singular set $S=S_1\cup S_2$ for the distribution $D$, and they project to uniform circular motion in the $xy$-plane.  
More precisely, let $q_0=q^i(0)$, and take any $n \in \mathbb Z$ and any $c_1, c_2 \in \mathbb R$ not both zero.  Then the curves are $(q^1(t), p^1(t))$ and $(q^2(t), p^2(t))$ are abnormal, where
\begin{align*}q^1(t)&=
\begin{pmatrix}
-\frac{\sqrt 5}{6}\cos(\frac{2}{5}t \pm \arctan 2+ \phi_0)+c_x\\  
-\frac{\sqrt 5}{6}\sin(\frac{2}{5}t \pm \arctan 2+ \phi_0)+c_y \\
    -\frac{3}{5}t + \phi_0\\
    t \\ t \pm 2\arctan 2 -2\pi n
\end{pmatrix},
\\
p^1(t)&=
\begin{pmatrix}
c_1 \\ c_2 \\
\frac{1}{6}[(-2c_1+c_2)\cos(\frac{2}{5}t+\phi_0)-(c_1+2c_2)\sin(\frac{2}{5}t+\phi_0)]\\
\frac{1}{36}[(-2c_1+11c_2)\cos(\frac{2}{5}t+\phi_0)-(11c_1+2c_2)\sin(\frac{2}{5}t+\phi_0)]\\
\frac{1}{36}[(-10c_1-5c_2)\cos(\frac{2}{5}t+\phi_0)+(5c_1-102c_2)\sin(\frac{2}{5}t+\phi_0)]
\end{pmatrix},\\
q^2(t)&=
\begin{pmatrix}
-\frac{1}{2}\cos(\frac{2}{3}t + \phi_0)+x_0+\frac{1}{2}\cos(\phi_0) \\  
-\frac{1}{2}\sin(\frac{2}{3}t + \phi_0)+y_0+\frac{1}{2}\sin(\phi_0) \\  
    -\frac{t}{3}+\phi_0\\ t \\ t-2\pi n
\end{pmatrix},\,
p^2(t)=
\begin{pmatrix}
0 \\ 0 \\ 6 \\ 1 \\ 1
\end{pmatrix},\,
\end{align*}
and 
$$
c_x = x_0 + \frac{\sqrt 5}{6}\cos(\pm\arctan 2+ \phi_0) \qquad \text{and} \qquad c_y = y_0 + \frac{\sqrt 5}{6}\sin(\pm\arctan 2+ \phi_0).
$$
%{\color{blue} $q^1$ disagrees with George's version -- someone else should check!}
\end{Proposition}

\begin{proof}
Note that $q^i(t)$ is just the horizontal lift of the naive parametrization of the lines which constitute the connected components of $S_i$ projected to the $\theta_1 \theta_2$-plane; these are the colored lines in Figure \ref{fig: sing sets}.  Thus for any time $t$ we have $q^i(t) \in S_i$.  The curves $q^i$ are integral curves for $F_1 +F_2$ and therefore horizontal.  In fact, $\text{span}\{F_1 +F_2\}$ is the intersection of $D$ and the set of Cauchy characteristics for $D+[D,D]$.  It is also the intersection of $D$ with the tangent bundle $TS$.

%First observe that every regular point for $D$ has small growth vector $(2,3,5)$.  The famous paper \cite{Cartan} shows that distributions of this type do not possess abnormal curves, thus any potential abnormal curves for the copepod system must lie within the singular set $S$.  Alternatively, Section 5.2 of \cite{SR book} shows that abnormal extremals are characteristic curves for $D^{\perp}$, which correspond to the intersection of $D$ and the set of Cauchy characteristics for $D+[D,D]$ restricted to the singular set for $D$.  Thus the abnormals are precisely the integral curves of $F_1 +F_2$ restricted to $S$.  

%{\color{blue} Need to carefully check/write everything above to be rigorous!  Can we find the specific result in Cartan?  Argue from Richard's book that abnormal iff characteristic?  Would need vector field version of characteristic.}

It is straightforward to check that the pairs $(q^1(t), p^1(t))$ and $(q^2(t), p^2(t))$ satisfy Hamilton's equations.  Observe that $p^i$ is constant in the first two components since our control vector fields do not depend on $x$ or $y$.  We also have by construction that 
$$\dot q^i = F_1 + F_2 |_{q^i} = \frac{\partial H_a}{\partial p}\bigg |_{(q^i,p^i)}.$$
Interestingly, the abnormal Hamiltonian also satisfies
$$\frac{\partial H_a}{\partial q}\bigg |_{(q^2,p^2)}=0.$$
Note that $p^1$ is actually a 2-parameter family of curves, and $p^2$ is constant and therefore an integral of motion.

It is similarly straightforward to check that $(q^i(t), p^i(t))$ satisfy  abnormal equations (\ref{eq: abnormal1} -- \ref{eq: abnormal4}).  Recall that $F_3|_{S_2}=0$ so equation (\ref{eq: abnormal3}) is satisfied for any $p$ on $S_2$.  Similarly, on any continuous curve within $S$ we have $F_4=-F_5$ and $u_1=u_2$, so equation (\ref{eq: abnormal4}) is satisfied for any $p$.
\end{proof}

\begin{Corollary}
Abnormal strokes contained in the singular set for $D$ can produce neither rotation nor displacement.
\end{Corollary}

\begin{proof}
The only possible abnormal stroke lying in $S$ would require the legs tracing a segment of one of the colored lines in Figure \ref{fig: sing sets} first forward and then backward.  The symmetry of such a stroke prevents any net rotation or displacement.
\end{proof}

%Note that this result stands in contrast to the translational copepod in \cite{optimal3}, where the stroke obtained by tracing out the edges of the triangle $0\leq\theta_1\leq\theta_2\leq \pi$ was shown to be the only (piecewise smooth) abnormal stroke.  This stroke indeed produced a net displacement for the model used there. {\color{blue} Monique, does this seem correct? Check legs!! Delete}

\begin{Proposition}\label{prop: vert horiz not abnormal}
Neither control can be zero along an abnormal extremal. 
\end{Proposition}

\begin{proof}
Without loss of generality, assume $u_2=0$ along an abnormal extremal $(q(t), p(t))$, so $\theta_2$ is constant. Assume $u_1\neq0$; otherwise the system is at a stationary point.
Equation (\ref{eq: abnormal4}) then implies that $\langle p,F_4(q)\rangle=0$. Differentiating this equation yields 
\begin{equation}
    u_1\langle p, F_6(q) \rangle + u_2\langle p, F_7(q) \rangle =0,
\end{equation}
which reduces to $\langle p, F_6(q) \rangle$ since $u_2=0$. A straightforward calculation shows that the vector fields $F_1,F_3,F_4,F_6$ are linearly independent along such curve. Moreover, they all are identically zero in the fifth component, so the only $p$ mutually orthogonal to these four vector fields is of the form $(0, 0, 0, 0, p_5)$.
Then using (\ref{eq: 2legF2}) and $\langle p,F_2\rangle=0$ we also have $p_5=0$.  Thus $p=0$, which contradicts the maximum principle. 
%We consider the case where $\theta_1$ is constant; the other case is similar.  Suppose $\theta_1=k$.  We can choose $t=\theta_2$, so that $u_1(t)=0$ and $u_2(t)=1$.  The horizontal condition (\ref{eq: Control System 2 legs}) then forces $\phi(t)=\frac{1}{12}\sin(t) - \frac{t}{4}+\phi(0)$; it also determines $x(t), y(t)$ but these are irrelevant.  I have an easier proof......
%Great! Let me draft it and you can double check
%Ok!
%{\color{blue} Someone check this argument.  Old incomplete proof is commented out.}
%For sake of contradiction, suppose $q(t)$ is the projection of an abnormal curve with $\theta_i$ constant in time.  Without loss of generality we take $i=2$, so that $\theta_2(t)=k$ and therefore $u_2(t)=0$.  Note if that $\theta_1$ is also constant then there is no motion and we discard this case.  Now since $q(t)$ is not contained in $S$, we have by Theorem \ref{thm: sing} that the vector fields $F_1, F_2, F_3, F_4$ are linearly independent along $q$.  But the abnormal equations (\ref{eq: abnormal1} -- \ref{eq: abnormal4}) combined with $u_2=0$ imply the existence of an adjoint vector $p(t)$ orthogonal to $F_1, F_2, F_3, F_4$...
%{\color{blue} Since $p$ is 5-dim, where is the contradiction here?}
\end{proof}

\begin{Corollary} 
The motion in Example \ref{ex: triangle} is not abnormal.  The boundary box (\ref{eq: constraint 2 legs}) is not abnormal.  
\end{Corollary}

\begin{Proposition}
The abnormal curves not lying in $S$ are solutions to the Hamiltonian system
$$ \dot q = \frac{\partial \tilde H_a}{\partial p},  \qquad  \dot p = -\frac{\partial \tilde H_a}{\partial q}
$$
for the Hamiltonian 
$$\tilde H_a(q,p) = \langle p, F_1 \rangle\langle p, F_5 \rangle - \langle p, F_2 \rangle\langle p, F_4 \rangle.
$$
\end{Proposition}

\begin{proof}
Note that Proposition \ref{prop: explicit abnormals} considers the abnormal curves within $S$.
Assume $(q(t),p(t))$ is an abnormal curve not lying within $S$ for any time interval.  By equations  (\ref{eq: abnormal1} -- \ref{eq: abnormal3}), we have that $p$ is orthogonal to $F_1(q), F_2(q), F_3(q)$.  But on the complement of $S$, we have that $F_1, F_2, F_3, F_4, F_5$ are linearly independent, so $p$ cannot also be orthogonal to both $F_4$ and $F_5$.  Without loss of generality, assume $p$ is not orthogonal to $F_4$.  Then we can solve equation (\ref{eq: abnormal4}) for 
$$ u_1 = -u_2 \frac{\langle p, F_5 \rangle}{\langle p, F_4 \rangle}
$$ obtaining 
$$ H_a = -u_2 \frac{\langle p, F_5 \rangle}{\langle p, F_4 \rangle}\langle p, F_1 \rangle + u_2 \langle p, F_2 \rangle.
$$
Thus we scale $H_a$ to obtain a new Hamiltonian in which the controls do not appear at all:

$$\tilde H_a(q,p) = \langle p, F_1 \rangle\langle p, F_5 \rangle - \langle p, F_2 \rangle\langle p, F_4 \rangle.
$$
As the abnormal equations have been satisfied by construction, any solution to Hamilton's equations for this Hamiltonian will indeed be an abnormal curve.
\end{proof}

%{\color{blue} Do we want to include any of the examples I numerically integrated?}{\color{red} I am inclined to say. I am not sure we have anything really insightful to say at this stage.}

\subsection{Normal Extremals}  \label{subsec: normal}
% Include simulations/pictures.  

%Elastica: conjecture/observe we seem to get this when asking bocop to let $\phi$ go from whatever to whatever, but periodic, with stroke.  The start/end position isn't clear -- need more simulations.  We probably won't prove, but here are some thoughts.  We could follow Richard et al.: encode entire problem as sR geod problem, compute geodesic/Hamilton equations, use these to compute curvature in first two components, show this satisfies some ODE for elastica.  This seems computationally hard, and not likely to work for other reasons: would have to somehow encode the conditions of the conjecture above, also the copepod vector fields (and momenta/covector fields) do not depend on $x,y$.  Alternatively we can try to write down one of the elastica ODEs in terms of an arbitrary parametrization (not necessarily by arclength).  Then we could at least check some examples numerically.  Not sure whether this discussion should go in Discussion and Conclusions.

Taking $p_0=-\frac{1}{2}$, our normal Hamiltonian is 
$$ H_n(q, p, u)= u_1\langle  p, F_1(q) \rangle + u_2\langle  p, F_2(q) \rangle -\frac{1}{2}(u_1^2 +u_2^2).
$$
We analyze the normal extremals indirectly, using the optimal control software \texttt{Bocop} (see Section \ref{sec: methods} and \cite{bocop1}).   Our investigations have led to two interesting observations regarding the behavior of normal extremals.

\begin{figure}%[H]
    \subfigure[]{\includegraphics[width=7 cm]{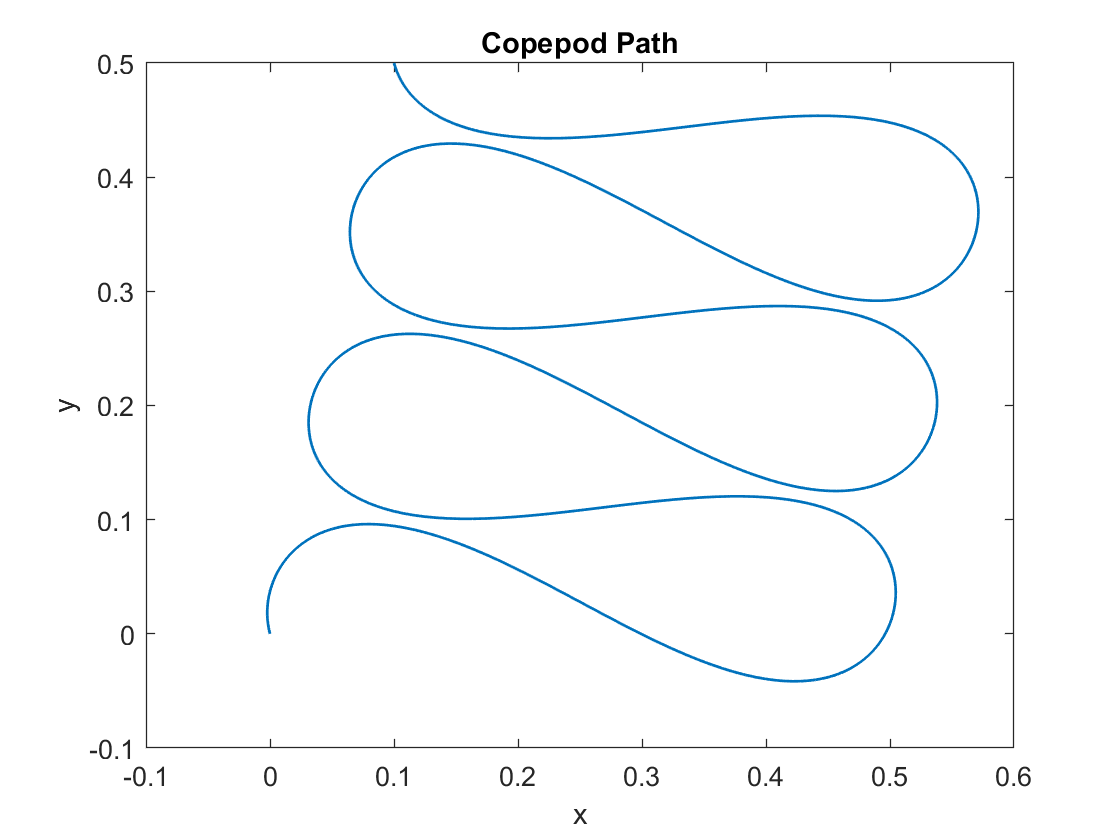}} 
    \subfigure[]{\includegraphics[width=7 cm]{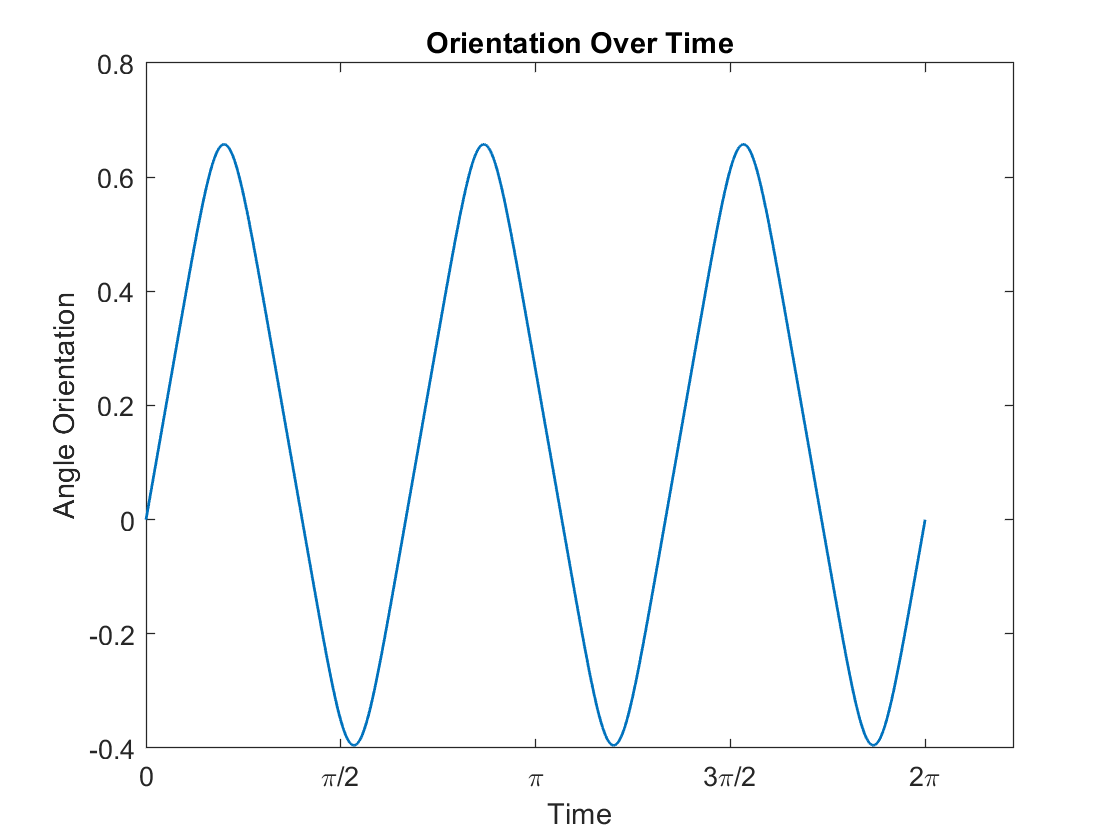}}
    \subfigure[]{\includegraphics[width=7 cm]{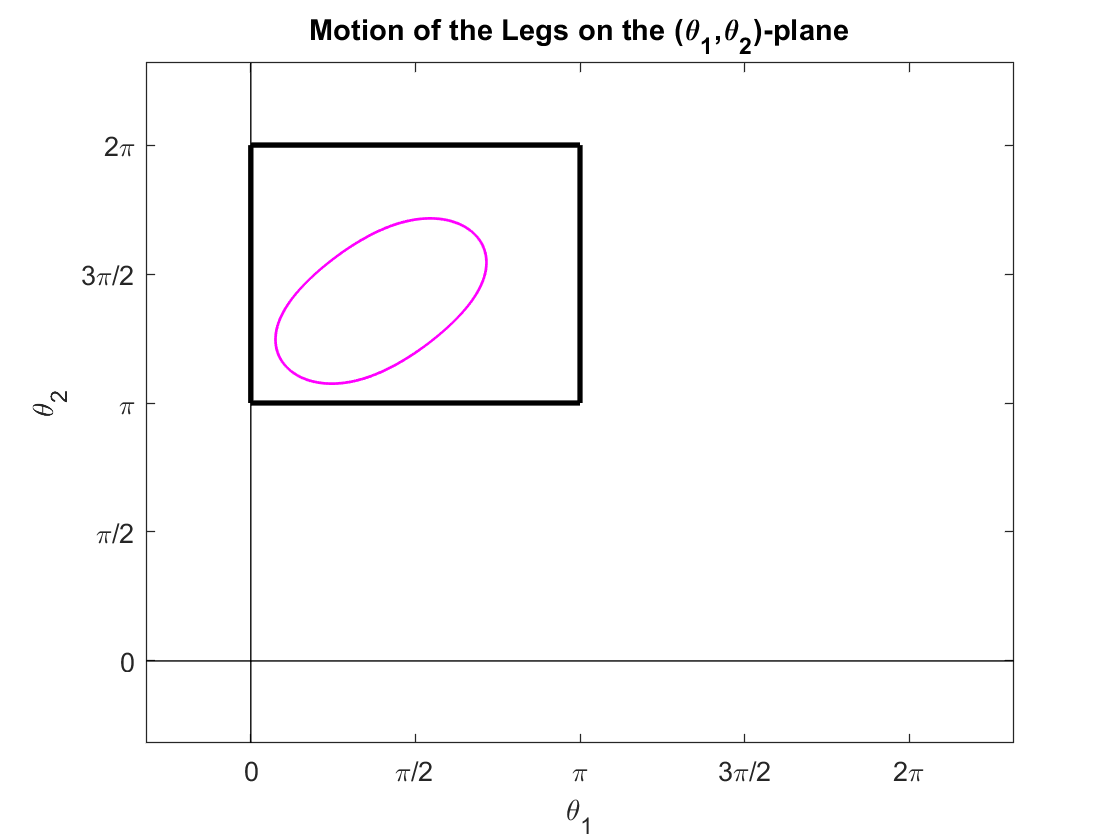}} 
    \subfigure[]{\includegraphics[width=7 cm]{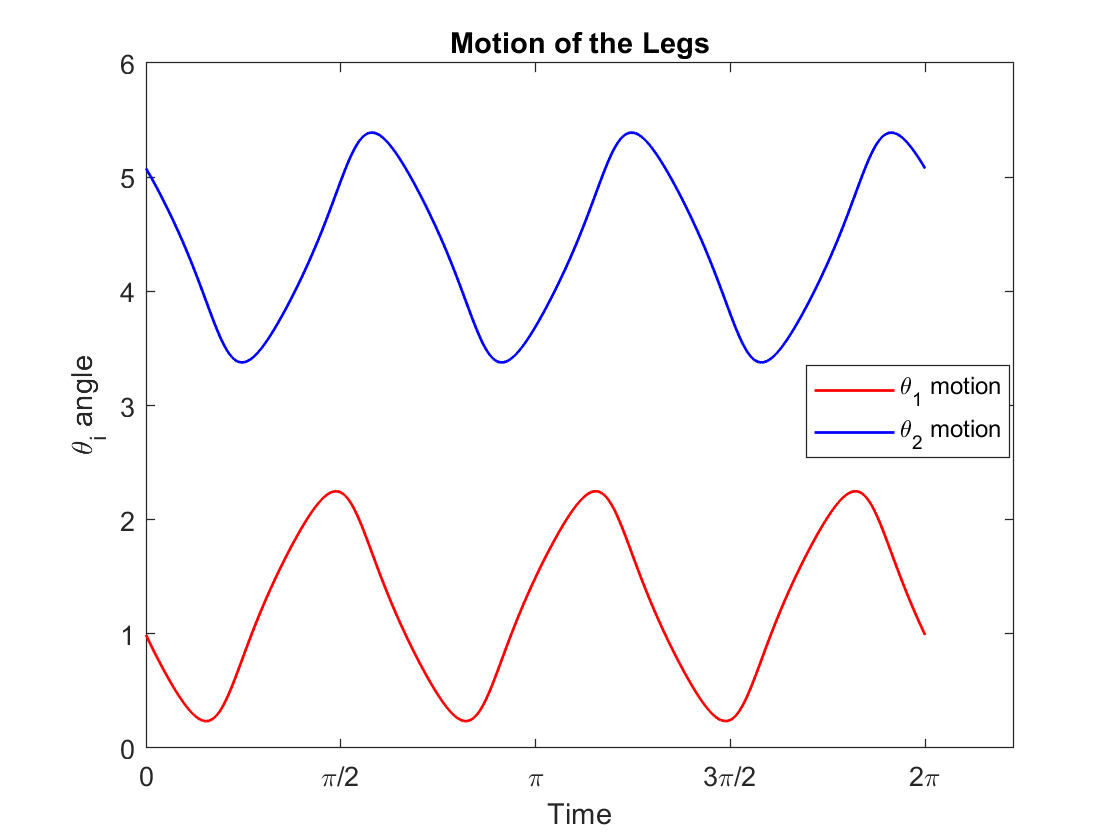}}
    \caption{A simulated optimal normal trajectory showing elastica-like motion in the plane.  Here we set the following boundary conditions: $(x,y)(0)=(0,0)$,  $(x, y)(2\pi)=(0.1,0.5)$, and  $(\phi, \theta_1, \theta_2)(0)=(\phi, \theta_1, \theta_2)(2\pi)$.  The energy here is approximately 54.311.
    % $(x(0), y(0))=(0,0)$,  $(x(2\pi), y(2\pi))=(0.1,0.5)$, $\phi(2\pi)=\phi(0)$,  $\theta_1(2\pi)=\theta_1(0)$,  $\theta_1(2\pi)=\theta_1(0)$. 
      (\textbf{a}) The elastica-like path of the microswimmer in the $xy$-plane.  (\textbf{b}) Orientation $\phi(t)$.  (\textbf{c}) The trajectory in the $\theta_1 \theta_2$-plane, within the constraint square (\ref{eq: constraint 2 legs}). (\textbf{d}) Angles of the legs $\theta_1(t)$ and $\theta_2(t)$.  
      \label{fig: elastica}}
\end{figure}

First, certain trajectories seem to show the copepod moving along a curve in the $xy$-plane which is a type of Euler elastica.  See Figure \ref{fig: elastica} for one example.  In particular,  we observe this behavior when fixing the start and end positions in the plane, demanding that the net change in orientation is zero, and demanding that the copepod completes a stroke, with no other imposed boundary conditions.
It can be observed that the legs follow a periodic motion, and in turn the orientation of the copepod is periodic as well. The motion in the angular phase plane $(\theta_1,\theta_2)$ is a perfect ellipsoid within the constraint space, reflecting the symmetry of the motion of the two legs.

%{\color{red} Also there is a symmetry between the legs, should we explicit that?}  {\color{blue} I think this is clear from the ellipsoid.}

A possible route to proving that this phenomena holds is suggested by \cite{elastica}.  Our optimal control problem can be translated into a geodesic problem in sub-Riemannian geometry.  Our control vector fields $F_1$ and $F_2$ have dual momenta $P_1=\langle p,F_1\rangle$ and $P_2=\langle p,F_2\rangle$, and the sub-Riemannian Hamiltonian $H_{sR}=\frac{1}{2}(P_1^2+P_2^2)$ generates normal geodesics corresponding to our normal optimal copepod trajectories.  These geodesics parametrized by arc length correspond to solutions of Hamilton's equations for $H_{sR}$ (geodesic equations) with energy $H = 1/2$.  These equations could potentially allow us to show that the curvature $\kappa$ of the projection $(x(t), y(t))$ satisfies one of the defining differential equations for Euler elastica.  The obstacles here are that the computations are unwieldy, and it is not clear how to impose the boundary conditions which seem to lead to elastica-like behavior in the copepod.

Our second interesting observation concerns the triangle $\mathcal T$, appearing in the lower right corner of the constraint square (\ref{eq: constraint 2 legs}), consisting of the boundary of the set $\{ (\theta_1, \theta_2) :\ 0 \leq \theta_1 \leq \theta_2 - \pi \leq \pi \}$.   Our simulations show that following this triangle is optimal for a copepod desiring to rotate a prescribed amount.  More precisely,  suppose we specify the net rotation $\Delta \phi$ but impose no other boundary conditions: we do not specify the start or end points in the plane, or require the motion be a stroke.  Then the optimal motion of the legs traces out the triangle $\mathcal T$ from the top right corner counterclockwise; it may go around $\mathcal T$ more than once,  not necessarily an integer number of times.  In fact, it will never go around an integer number of times (which would constitute a stroke).  Our observations suggest the following characterization of the motion:
$$
\Delta \phi \in \begin{cases}
[0, \frac{2\pi}{3}] \qquad  \text{just follow hypotenuse: 0 to .5 times around}\ \mathcal{T} \\
(\frac{2\pi}{3}, \frac{5\pi}{6}] \quad \, \text{once around $\mathcal{T}$,  then hypotenuse: 1 to 1.5 times around}\ \mathcal{T}\\
(\frac{5\pi}{6}, \pi] \quad  \ \  \text{twice around $\mathcal{T}$, then hypotenuse: 2 to 2.5 times around}\ \mathcal{T}.\\
\end{cases}
$$
In any of these cases, the hypotenuse need not be traced out completely.  For example, to rotate $\pi/3$ radians one would simply traverse half the hypotenuse then stop.  
See Figure \ref{fig: triangle 2.5} for an example of the third case with $\Delta \phi = \pi$.
Note the symmetry of the legs in Figure \ref{fig: triangle 2.5}(d), which is implicit in the triangle $\mathcal T$ itself.
Of course, to rotate more than $\pi$ radians one simply reverses this process (starting at the bottom left and following $\mathcal T$ clockwise).   %By Proposition \ref{prop: vert horiz not abnormal}, none of these motions are abnormal.  

\begin{figure}%[H]
    \subfigure[]{\includegraphics[width=7 cm]{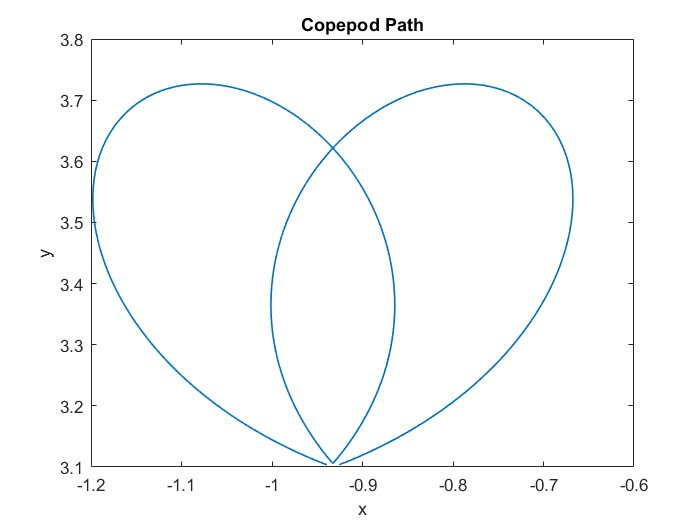}} 
    \subfigure[]{\includegraphics[width=7 cm]{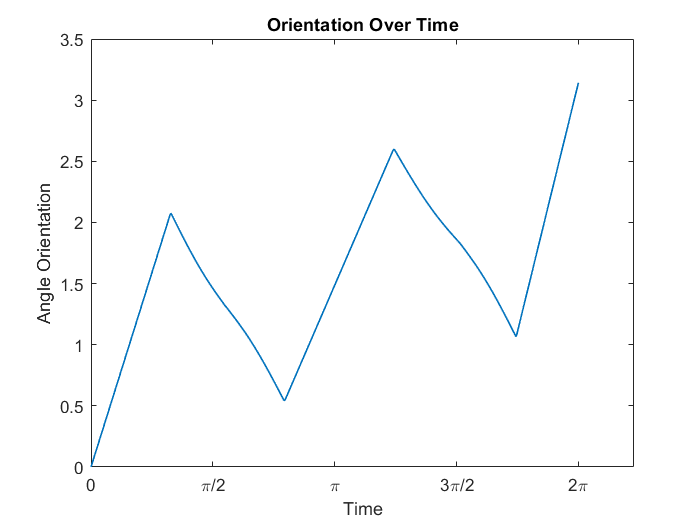}}
    \subfigure[]{\includegraphics[width=7 cm]{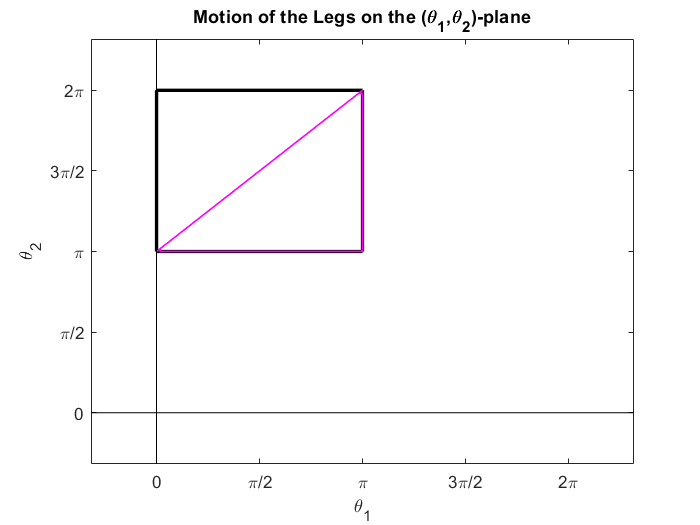}} 
    \subfigure[]{\includegraphics[width=7 cm]{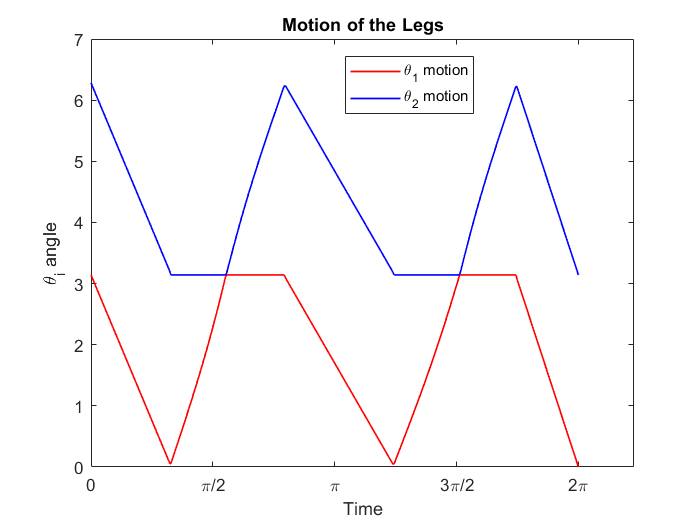}}
    \caption{A simulated normal trajectory showing that following the triangle $\mathcal T$ is optimal for creating rotation.   Our only boundary conditions concern the orientation angle:  $\phi(0)=0$ and $\phi(2\pi)=\pi$. The energy here is approximately 107.735.
      (\textbf{a}) The path of the animal in the $xy$-plane.  (\textbf{b}) Orientation $\phi(t)$.  (\textbf{c}) The trajectory in the $\theta_1 \theta_2$-plane, within the constraint square (\ref{eq: constraint 2 legs}).  The path traces out the triangle $\mathcal T$ exactly 2.5 times counterclockwise starting from the upper right vertex. (\textbf{d}) Angles of the legs $\theta_1(t)$ and $\theta_2(t)$. %{\color{red} do we have the energy?}
      \label{fig: triangle 2.5}}
\end{figure}

Note that traveling along the hypotenuse induces no displacement and traversing the complete triangle induces very small net displacement (see Example \ref{ex: triangle}).  Thus these motions represent optimal swimming for a copepod attempting to rotate any amount without much net displacement.
%{\color{red} it does moves in the plane quite a lot actually during the motion so maybe we should say something in that regard}.  
Any rotation amount less than or equal to $2\pi/3$ can be achieved optimally with no displacement at all.  Intuitively, this demonstrates the fact that traveling along the hypotenuse gives the strongest possible power stroke for inducing rotation -- the legs of the triangle are simply necessary to move the copepod legs back into position for another power stroke in a way that minimizes backwards rotation. A motion which includes the legs of the triangle, as in Figure \ref{fig: triangle 2.5}, does require the copepod to move around in the plane, but it returns to nearly its original position.  

In Figure \ref{fig: gallery} we provide a catalog of the type of topological curves in the $xy$-plane obtained when varying the boundary conditions.   The boundary conditions themselves appear in Table \ref{table}.

\begin{figure}%[H]
    \subfigure[]{\includegraphics[width=4.69 cm]{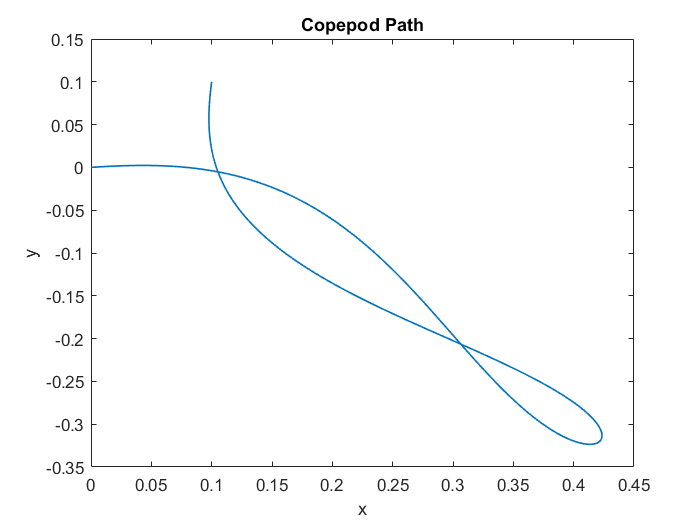}} 
        \subfigure[]{\includegraphics[width=4.69 cm]{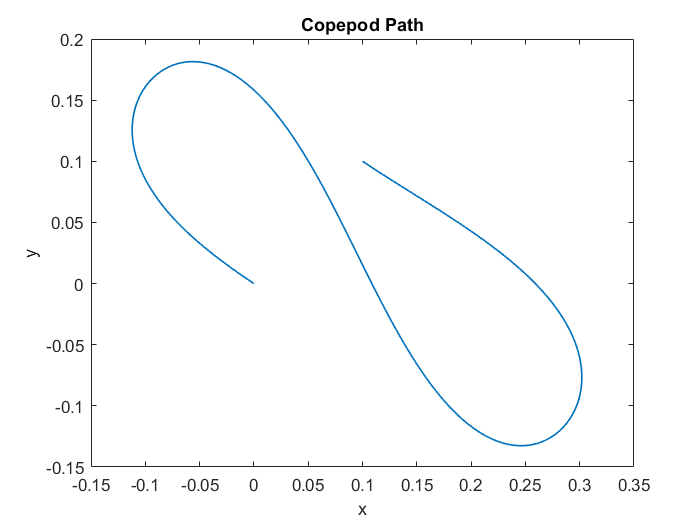}} 
            \subfigure[]{\includegraphics[width=4.69 cm]{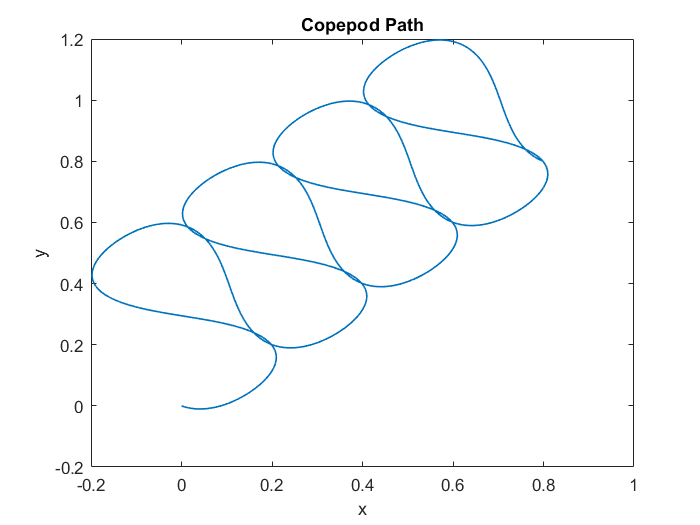}} 
               \subfigure[]{\includegraphics[width=4.69 cm]{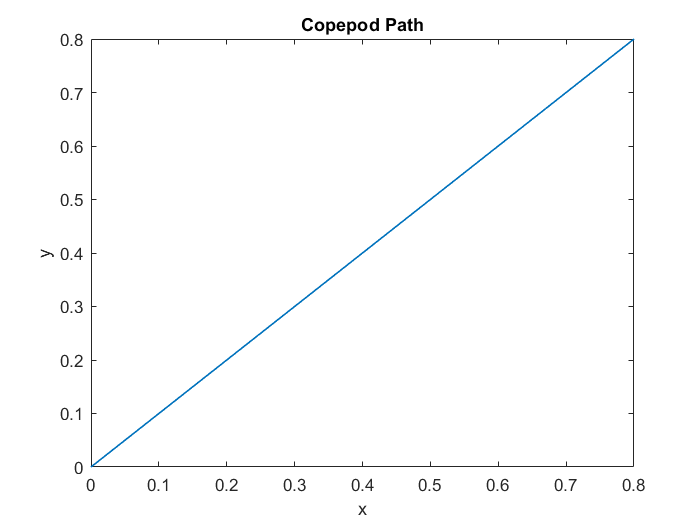}} 
        \subfigure[]{\includegraphics[width=4.69 cm]{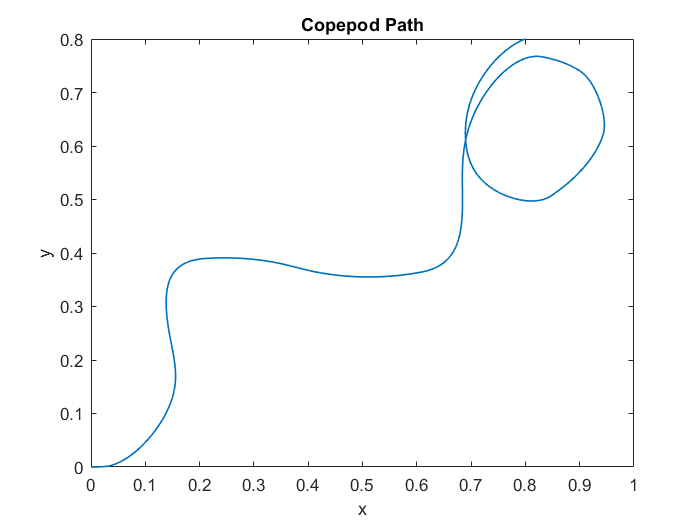}} 
            \subfigure[]{\includegraphics[width=4.69 cm]{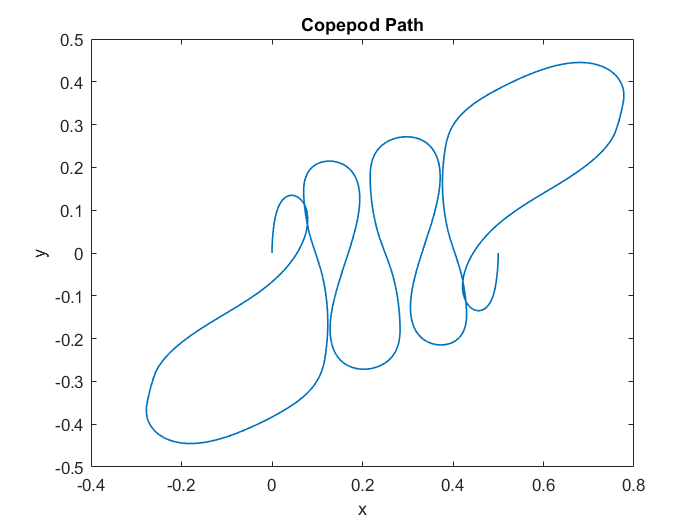}} 
               \subfigure[]{\includegraphics[width=4.69 cm]{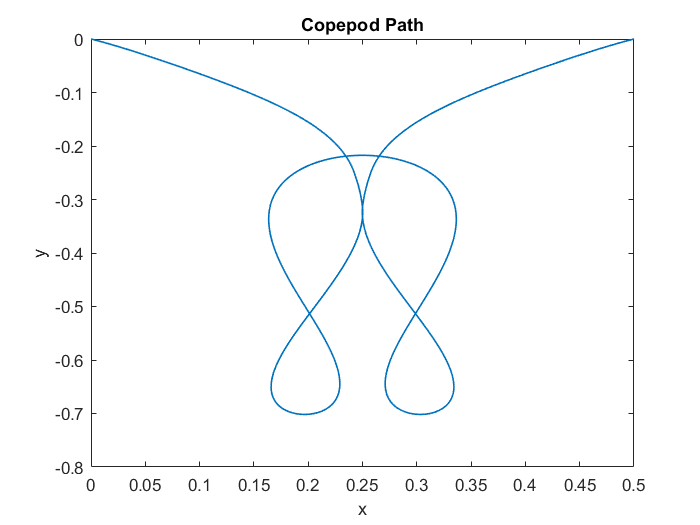}} 
        \subfigure[]{\includegraphics[width=4.69 cm]{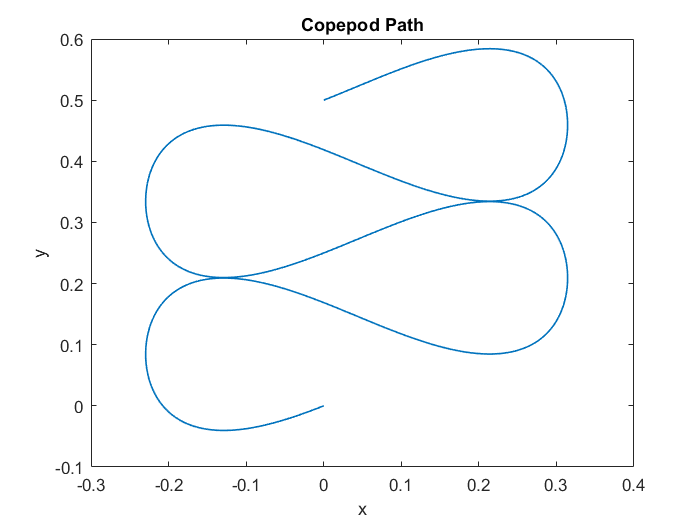}} 
            \subfigure[]{\includegraphics[width=4.69 cm]{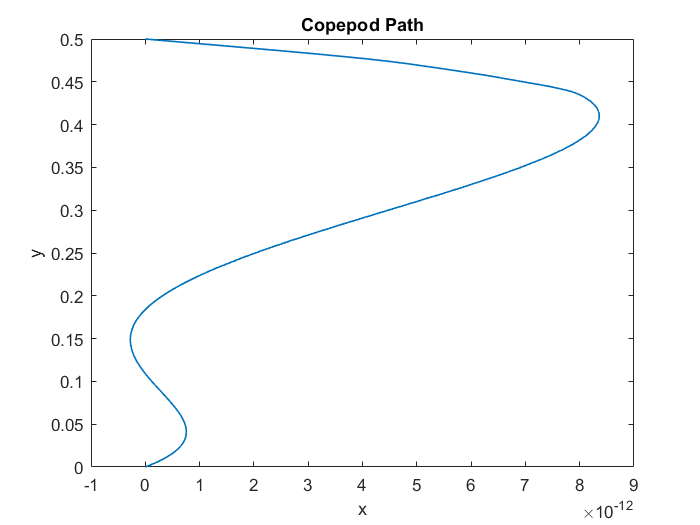}} 
               \subfigure[]{\includegraphics[width=4.69 cm]{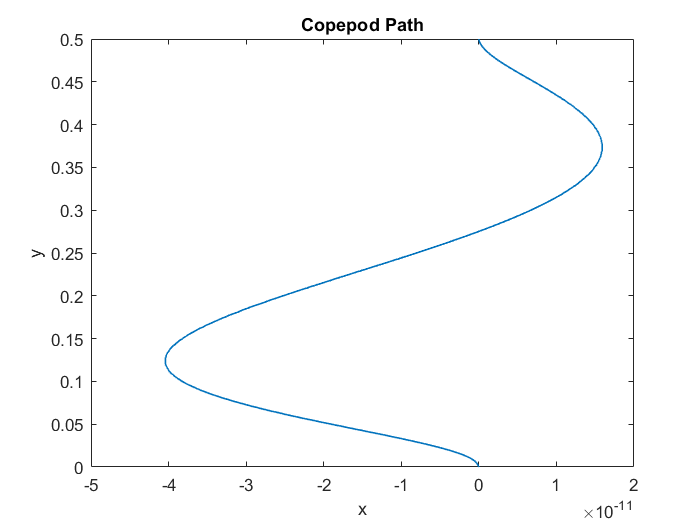}} 
        \subfigure[]{\includegraphics[width=4.69 cm]{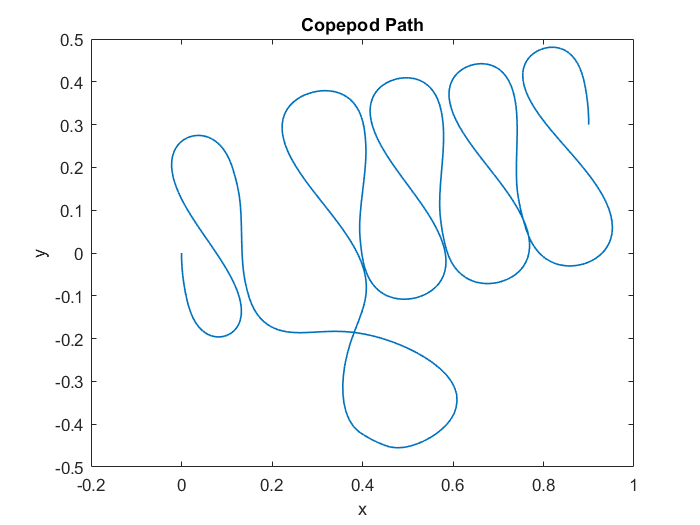}} 
            \subfigure[]{\includegraphics[width=4.69 cm]{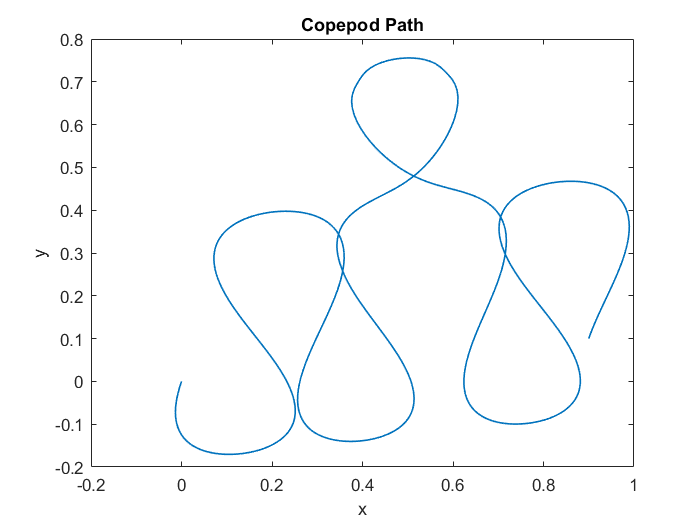}} 
               \subfigure[]{\includegraphics[width=4.69 cm]{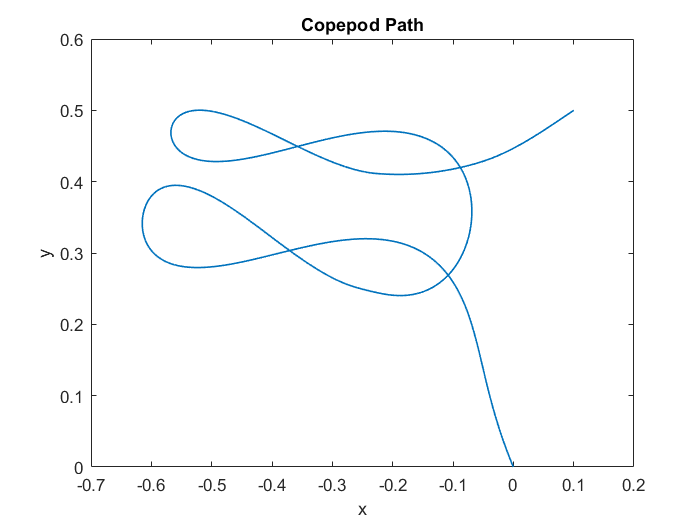}} 
        \subfigure[]{\includegraphics[width=4.69 cm]{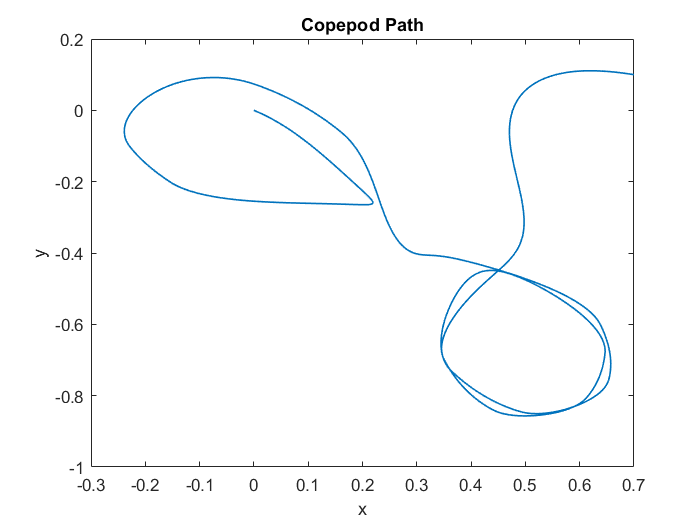}} 
            \subfigure[]{\includegraphics[width=4.69 cm]{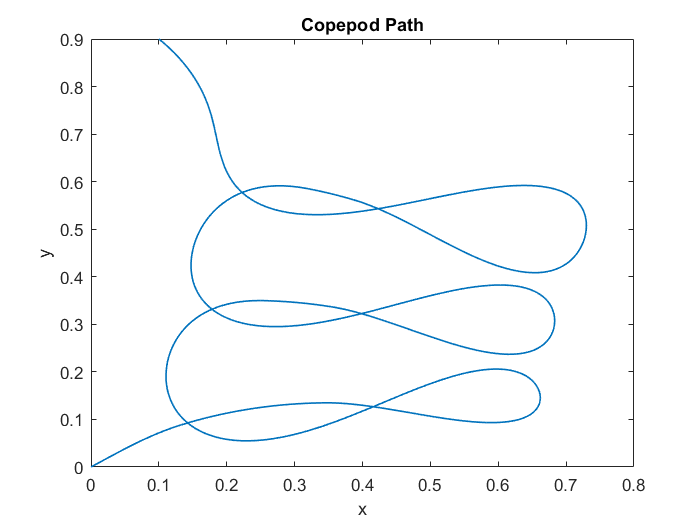}} 
    \caption{Gallery of simulated normal extremals showing copepod paths in the $xy$-plane.   In all cases the copepod begins at the origin $(0,0)$ and $t\in [0, 2\pi]$.   Additionally,  we specify boundary conditions for $x(2\pi), \, y(2\pi), \, \phi(0)$, and $\phi(2\pi)$; these are given in Table \ref{table}.   No other boundary conditions are imposed.  %{\color{blue} Need to figure out formatting, maybe put this at the end. When totally done we should resize/place figures to minimize white space.}  
      \label{fig: gallery}}
\end{figure}

% The MDPI table float is called specialtable
\begin{table}[H] 
\caption{The boundary conditions corresponding to the paths in Figure \ref{fig: gallery}.  In all cases the initial position is $(0,0)$.  In some simulations certain conditions were unspecified, labeled ``Free" here.  For some simulations we impose the stroke condition that $\theta_i(0)=\theta_i(2\pi)$. The total energy of each path is not imposed, but listed here for comparison.
\label{table}}
%{\color{red} could we add a column to provide the energy? for instance it woul dbe interesting to compare c) and d)}
%%% \tablesize{} %% You can specify the fontsize here, e.g., \tablesize{\footnotesize}. If commented out \small will be used.
\begin{tabular}{|c|c|c|c|c|c|}
%\toprule
 \hline 
\textbf{Subfigure}	& \textbf{Final position} %$\mathbf{(x(2\pi), y(2\pi))}$
& \textbf{Initial orientation}  & \textbf{Final orientation} & \textbf{Stroke} & \textbf{Energy}\\
%\midrule
 \hline 
(a)		& $(0.1, 0.1)$  & Free & Free & Yes & 3.868 \\  \hline 
(b)		& $(0.1, 0.1)$  & Free & $\phi(0)$ & Yes & 4.670 \\ \hline 
(c)		& $(0.8, 0.8)$  & Free & $\phi(0)$ & Yes & 215.560 \\ \hline 
(d)	    & $(0.8, 0.8)$  & Free & $\phi(0)$ & No & 1.895\\ \hline 
(e)	    & $(0.8, 0.8)$  & 0 & $0$ & No & 19.968 \\ \hline 
(f)		& $(0.5, 0)$  & 0 & 0 & Yes & 122.723\\ \hline 
(g)		& $(0.5, 0)$  & Free & Free & Yes & 27.745 \\ \hline 
(h)		& $(0, 0.5)$  & Free & $\phi(0)$ & Yes & 43.683 \\ \hline 
(i)		& $(0, 0.5)$  & Free & $\phi(0)$ & No & 0.319\\ \hline 
(j)		& $(0, 0.5)$  & Free & Free & No & 0.319 \\ \hline 
(k)		& $(0.9, 0.3)$  & 0 & 0 & Yes & 197.796\\ \hline 
(l)		& $(0.9, 0.1)$  & Free & $\phi(0)$ & Yes & 146.854 \\ \hline 
(m)		& $(0.1, 0.5)$  & 0 & Free & Yes & 30.321 \\ \hline 
(n)		& $(0.7, 0.1)$  & 0 & Free & Yes & 80.652 \\ \hline 
(o)		& $(0.1, 0.9)$  & Free & 0 & Yes & 91.949 \\ \hline 
%\bottomrule
\end{tabular}
\end{table}

%%%%%%%%%%%%%%%%%%%%%%%%%%%%%%%%%%%%%%%%%%
%\section{Discussion}

%Authors should discuss the results and how they can be interpreted from the perspective of previous studies and of the working hypotheses. The findings and their implications should be discussed in the broadest context possible. Future research directions may also be highlighted.

%%%%%%%%%%%%%%%%%%%%%%%%%%%%%%%%%%%%%%%%%%
\section{Discussion and Conclusions} \label{sec: discussion}

Here we have provided a mathematical model of a swimming copepod nauplius with two legs moving in a plane.  This model allows for both rotation and two-dimensional displacement by periodic deformation of the swimmer's body.
The system was studied from the framework of optimal control theory, with a simple cost function designed to approximate the mechanical energy expended by the copepod.  We have found that this model is sufficiently realistic to recreate behavior similar to those of observed copepod nauplii, yet much of the mathematical analysis is tractable.  In particular, we have shown that the system is controllable, but there exist singular configurations where the degree of non-holonomy is non-generic.  We have also partially characterized the abnormal extremals and provided explicit examples of families of abnormal curves. Finally, we have numerically simulated normal extremals and observed some interesting and surprising phenomena.

This work suggests a plethora of interesting open problems and directions for future research.  First, there are a number of potential generalizations and modifications to our model which may lead to even more realistic behaviors.  For example, one can study the model with four or six legs coupled with the appropriate constraints.  Real copepods have six legs. Further, this model has the potential to design soft small-scale synthetic robots \cite{optimal3,Hu, Sitti}.   Alternatively, or perhaps additionally, one could work in a three-dimensional environment, which is obviously more realistic.  Moreover, there are other reasonable cost functions to consider, including more complicated versions of mechanical energy.  Instead,  it may be that copepods seek to minimize the time needed to perform a given motion, or the total distance traveled, either to evade predators or capture prey more effectively.  

Without generalization, our current model already offers ideas for future research.  In particular, it would be quite interesting to find a mathematical, physical, or biological explanation for the observed elastica-like paths, such as the one shown in Figure \ref{fig: elastica}.  Even numerical verification that these paths are indeed forms of elastica would be worth pursuing.    A potential approach is described in Section \ref{subsec: normal}, which in turn leads to other questions of a differential geometric flavor.  Our optimal control problem can indeed be cast as the geodesic problem for a particular sub-Riemannian geometry, which appears geometrically interesting.  The vast sub-Riemannian literature may yield geometric or metric tools providing deeper insight into the copepod system.  

While the experimental approach in Section \ref{subsec: normal} led to some interesting observations, the normal extremals are still largely not understood.  It would be particularly interesting to explore path-planning for the copepod system.  It is also important to recognize that \texttt{Bocop}, like any mathematical software, has limitations, some of which we encountered.  In particular, this is local optimization software, and we are not working on a convex optimization problem with one global extremum.  Thus, despite the symmetries of the problem, the numerical results were sensitive to transversality conditions (for example, specifying that copepod start at the origin).
%{\color{blue} Say something here about the issues we had with bocop.  Be honest.  Some things seem inconsistent with the obvious symmetries of the system, but some may be explainable from a purely control theoretic perspective.   UPDATE:  I think I understand the phenomena we saw, so just leave this out?  Maybe just add a sentence about local vs global optimization.}

Finally, we consider how well our results approximate actual observations of copepods in motion. In \cite{Bradley} the authors discuss copepod swimming and escape behavior, based on observations of their swimming patterns and activity. In particular, Figure 4 in that paper shows a helical pattern projecting onto the $xy$-plane like an ellipse. This correlates with our motion presented in Figure \ref{fig: triangle 2.5}. Figure 8 in \cite{Bradley} depicts escape trajectories for nauplii and copepodis which follow helical patterns that project on the $xy$-plane as Euler elastica. In \cite{Bruno}, the authors observe the positions of the appendages during prey capture and prey handling; in their Figure 5 we see the leg motions are oscillatory and mostly periodic, as in our simulations, during the prey handling phase. The most striking comparison comes with the observed behavior provided in \cite{Paffenhofer}. Indeed, the projection of the 3D swimming motion of the nauplii and early copepodid in their observations provide a similar complexity to our simulated trajectories.  Compare our Figure \ref{fig: gallery} to Figures 1 through 9 in their paper. It is quite remarkable that despite the simplified assumptions made on the number of legs and the cost, our results still capture the essence of swimming behavior for copepods. 

%Include elastica (numerical verification of example? intuition a la George?  potential to prove as in Montgomery et al?), generalizations (more legs, 3D), experiments/observations?, find abnormals, other cost functions (eg time), subR?, Niimoto observations, conserved quantities/symmetries (obvious ones but also $p^2$), 

%Limitations of bocop, things probably wrong, doesn't respect the obvious symmetries of the problem.  Most obviously the one from Jonas's email: imposing $\phi$ from 0 to $5\pi/6$ gives very different result from just imposing $\Delta\phi=5\pi/6$.  Similarly going from origin to $(0, 0.5)$ gives very different result from going to $(0.5, 0)$ with all other things equal.

%{\color{blue} TODO: `Translation' vesus `displacement' consistency.  Figure out why line numbers inconsistent.  Format and alphabetize biblio, delete unused references.}

%{\color{blue} TODO: Formatting.  Proofreading.}


\begin{thebibliography}{999}

\bibitem{SR geo}
A. Agrachev, D. Barilari, U. Boscain,
\textit{A Comprehensive Introduction to sub-Riemannian Geometry}, Cambridge University Press, 2019.

\bibitem{kinematics}
C.M. Andersen-Borg , E. Bruno, T. Ki\o rboe, 
{The kinematics of swimming and relocation jumps in copepod nauplii}, \textit{PLoS ONE}, \textbf{7} (2012), e47486.

\bibitem{elastica}
A. Ardentov, G. Bor, E. Le Donne, R. Montgomery, Y. Sachkov, Bicycle paths, elasticae and sub-Riemannian geometry, arXiv:2010.04201, (2021).

\bibitem{optimal1}
P. Bettiol, B. Bonnard, J. Rouot,
{Optimal strokes at low Reynolds number:
a geometric and numerical study of copepod and Purcell swimmers},
\textit{SIAM Journal on Control
and Optimization}, \textbf{56} (2018), 1794--1822.


\bibitem{book2}
B. Bonnard, M. Chyba, \textit{Singular Trajectories and Their Role in Control Theory}, Springer, 2003.

\bibitem{book}
B. Bonnard, M. Chyba, J. Rouot, \textit{Geometric and Numerical Optimal Control: Application to Swimming at Low Reynolds Number and Magnetic Resonance Imaging}, Springer, 2018.

\bibitem{optimal2}
B. Bonnard, M. Chyba, J. Rouot, D. Takagi,
{A numerical approach to the
optimal control and efficiency of the copepod swimmer},
\textit{2016 IEEE 55th Conference on Decision and Control}, (2016), 4196--4201. 

\bibitem{optimal3}
B. Bonnard, M. Chyba, J. Rouot, D. Takagi,
{Sub-Riemannian geometry, Hamiltonian dynamics, micro-swimmers, copepod nauplii and copepod robot}, \textit{Pac. J. Math. Ind.}, \textbf{10} (2018), 1--27.

\bibitem{Bradley}
C. J. Bradley, J. R. Strickler, E. J. Buskey, P. H. Lenz,
{Swimming and escape behavior in two species of calanoid copepods from nauplius to adult}, \textit{Journal of Plankton Research}, \textbf{35} (2013), 49--65.

\bibitem{Bruno}
E. Bruno, C. M. Andersen Borg, T. Ki\o rboe, Prey detection and prey capture in copepod nauplii, \textit{PLoS ONE}, \textbf{7} (2012), e47906. 

\bibitem{devine}
D. Devine, 
\textit{Locomotion and rotation with three stiff legs at low Reynolds number}, Master's thesis, University of Hawai'i, Manoa, 2016. 

\bibitem{Dreyfus}
R. Dreyfus, J. Baudry, H.A. Stone, Purcell’s “rotator”: Mechanical rotation at low Reynolds number, \textit{ Eur. Phys. J.  B}, \textbf{47} (2005), 161--164.

\bibitem{nanoswimmer}
W. Gao,  et al., {Cargo-towing fuel-free magnetic nanoswimmers for targeted drug delivery}, \textit{Small}, \textbf{8} (2012), 460--467.

\bibitem{Hu}
W. Hu, G. Lum, M. Mastrangeli, M. Sitti, Small-scale soft-bodied robot with multimodal locomotion, \textit{Nature}, \textbf{554} (2018), 81--85.

\bibitem{Jalali}
M.A. Jalali, M. R. Alam, S.Mousavi, Versatile low-Reynolds-number swimmer with three-dimensional maneuverability. \textit{Phys. Rev. E }, \textbf{90} (2014), 053006.

\bibitem{Lenz}
P. Lenz, D. Takagi, D. Hartline, 
Choreographed swimming of copepod nauplii,
\textit{J.  R.  Soc.  Interface},  \textbf{12} (2015), 20150776.

\bibitem{Liberzon}
D. Liberzon, 
\textit{Calculus of Variations and Optimal Control Theory: A Concise Introduction}, Princeton University Press, 2012. 

\bibitem{SR book}
R. Montgomery, \textit{A Tour of Subriemannian Geometries, Their Geodesics and Applications}, American Mathematical Society, 2000.

\bibitem{rotational}
K.T.M. Niimoto, K.J. Kuball, L.N. Block,  P.H. Lenz, D. Takagi,
{Rotational maneuvers of copepod nauplii at low Reynolds number},
\textit{Fluids}, \textbf{5} (2020).

\bibitem{Paffenhofer}
G.-A. Paffenh{\"o}fer, J.R. Strickler, K.D. Lewis, S. Richman, Motion behavior of nauplii and early copepodid stages of marine planktonic copepods, \textit{Journal of Plankton Research}, \textbf{18} (1996), 1699--1715.

\bibitem{Pontryagin}
L.S. Pontryagin, V.G. Boltyanskii, R.V. Gamkrelidze,  E.F. Mishchenko, \textit{The Mathematical Theory of Optimal Processes}, 1962.

\bibitem{purcell}
E.M. Purcell,
{Life at low Reynolds number},
\textit{American Journal of Physics}, \textbf{45} (1977), 3--11.

\bibitem{Rizvi}
M.S. Rizvi, A. Farutin, C. Misbah, Three-bead steering microswimmers, \textit{Phys. Rev. E}, \textbf{97} (2018), 023102.

\bibitem{Robinson}
H.E. Robinson, J.R. Strickler, M.J. Henderson, D.K. Hartline, P.H. Lenz, Predation strategies of larval clownfish capturing evasive copepod prey, \textit{Mar. Ecol. Prog. Ser.}, \textbf{614} (2019), 125--146.

\bibitem{Sitti}
M. Sitti, Miniature soft robots — road to the clinic, \textit{Nature Reviews Materials}, \textbf{3} (2018), 74--75.

\bibitem{takagi}
D. Takagi, Swimming with stiff legs at low Reynolds number,
\textit{Physical Review}, \textbf{92} (2015), 023020.

\bibitem{bocop1}
Team Commands, Inria Saclay, \textit{BOCOP: an open source toolbox for optimal control}, \url{http://bocop.org}, 2017.

\bibitem{copepod}
J. Turner, The importance of small planktonic copepods and their roles in pelagic marine food webs, \textit{Zoological Studies}, \textbf{43} (2004), 255--266. 



\end{thebibliography}
\end{document}